\documentclass[10pt,a4paper,draft]{article}
\usepackage{bbm}
\usepackage{pifont}
\usepackage{bbding}
\usepackage{amsmath, esint}
\usepackage{mathrsfs}
\usepackage{amsfonts}
\usepackage{amssymb}
\usepackage{amsfonts,amssymb,amsmath,indentfirst,cases,amsthm}
\usepackage{epsfig}
\usepackage[numbers,sort&compress]{natbib}
\allowdisplaybreaks
\def\ess~sup{\mathop{\rm ess~sup}}

\numberwithin{equation}{section}

\newenvironment{key words}{\emph{\texttt{Keywords}}\mbox{  }}{ }
\newtheorem{theorem}{Theorem}[section]
\newtheorem{lemma}[theorem]{Lemma}

\newtheorem{proposition}[theorem]{Proposition}

\newtheorem{definition}[theorem]{Definition}

\theoremstyle{remark}

\theoremstyle{plain}

\makeatletter

\newcommand{\Rmnum}[1]{\expandafter\@slowromancap\romannumeral #1@}
\makeatother
\begin{document}

 \title{\textbf{Higher Weak Differentiability to Mixed Local and Nonlocal Degenerate Elliptic Equations in the Heisenberg Group
}}
\author{Junli Zhang$^1$\thanks{Corresponding
author's E-mail:jlzhang2020@163.com}, Pengcheng Niu$^2$\\
\small{1. School of Mathematics and Data Science, Shaanxi University of Science and Technology,}\\
\small{ Xi'an, Shaanxi, 710021, P. R. China}\\
\small{2. School of Mathematics and Statistics, Northwestern Polytechnical
University,}\\
\small{ Xi'an, Shaanxi, 710129, P. R. China}
}
\date{}
\maketitle

\maketitle {\bf Abstract}\
In this paper, we investigate the higher weak differentiability of solutions to a class of mixed local and nonlocal degenerate elliptic equations in the Heisenberg group $\mathbb{H}^n$. Owing to the non-commutative property and two-step nilpotent Lie algebra structure of $\mathbb{H}^n$, we first employ an iterative scheme involving fractional difference quotients to establish the weak differentiability of solutions in the vertical direction. This is subsequently extended to the horizontal and vertical gradients. Then, by coupling a truncation argument with the difference quotient method, we prove the higher weak differentiability of the gradients of solutions.


\textbf {AMS} Primary: 35B65, 35D30; Secondary: 35J92, 35R11.
\def\Xint#1{\mathchoice
    {\XXint\displaystyle\textstyle{#1}}%
    {\XXint\textstyle\scriptstyle{#1}}%
    {\XXint\scriptstyle\scriptscriptstyle{#1}}%
    {\XXint\scriptscriptstyle\scriptscriptstyle{#1}}%
    \!\int}
\def\XXint#1#2#3{{\setbox0=\hbox{$#1{#2#3}{\int}$}
    \vcenter{\hbox{$#2#3$}}\kern-.5\wd0}}
\def\dashint{\Xint-}
\section{Introduction}

In recent years, the regularity theory for mixed local and nonlocal elliptic operators in the Euclidean setting has attracted significant attention. Notable advancements include interior Sobolev and Lipschitz-type boundary regularity \cite{BDVV22}, $C^{1,\alpha}$ regularity for weak solutions \cite{SVWZ22}, gradient estimates \cite{DM24}, and higher H\"{o}lder regularity \cite{GL23}. Furthermore, boundary behavior \cite{BMS23} and the analysis of problems with nonstandard growth-including local boundedness and Harnack inequalities \cite{DFZ24} have been extensively documented.

Despite these developments, the regularity theory for mixed local and nonlocal degenerate elliptic equations within the sub-Riemannian framework of the Heisenberg group $\mathbb{H}^n$ remains largely nascent. Addressing this gap, Oza and Tyagi \cite{OT25} recently established the existence and H\"{o}lder regularity of viscosity solutions for a class of mixed fully nonlinear operators in $\mathbb{H}^n$. This line of research was further advanced by Zhang and Niu \cite{ZN26}, who achieved a significant breakthrough regarding weak solutions to mixed local and nonlocal degenerate elliptic equations. Their work established fundamental regularity properties, including the local boundedness of weak subsolutions, H\"{o}lder continuity, and, crucially, both the Harnack and weak Harnack inequalities. More recently, Zhang \cite{Z26} further refined these results by establishing the $C^{1,\alpha}$ regularity for weak solutions to a class of mixed local and nonlocal operators on $\mathbb{H}^n$. This result provides a sharp characterization of the smoothness of solutions, extending the regularity theory beyond the H\"{o}lder continuous regime to the level of horizontal gradients.

In this paper, we investigate the following mixed local and nonlocal degenerate elliptic equation:
\begin{equation}\label{eq11}
  - \Delta_{\mathbb{H}^n} u\left( x \right) + {c_{Q,s}}P.V.\int_{{\mathbb{H}^n}} {\frac{{u\left( x \right) - u\left( y \right)}}{{\left\| {{y^{ - 1}} \circ x} \right\|_{{\mathbb{H}^n}}^{Q + 2s}}}dy}  = f\left( x \right),\;x \in \Omega
\end{equation}
where $0 < s < 1$, $Q = 2n + 2$, and $f \in L^2(\Omega)$ on a bounded domain $\Omega \subset \mathbb{H}^n$. In contrast to the theory of mixed Euclidean operators (e.g., \cite{DM24}, \cite{GL23}), where standard difference quotients can be directly applied to the weak formulation, establishing higher regularity in $\mathbb{H}^n$ presents fundamental obstacles stemming from its underlying non-commutative property and two-step nilpotent Lie algebra structure. Specifically, the non-trivial commutator relation between the horizontal vector fields $X_i$ and the vertical vector field $T$ strictly precludes a direct differentiation of the equation \eqref{eq11}. Furthermore, the nonlocal integral term inherently obstructs standard local iterative techniques. While the $C^{1,\alpha}$-regularity established by Zhang \cite{Z26} have successfully characterized the optimal H\"older continuity of horizontal gradients, the problem of higher weak differentiability (e.g., $HW^{m+2,2}$-regularity) for such mixed systems has remained entirely open.

To overcome these dual hurdles, we devise a novel two-stage strategy. We first implement an iterative scheme based on fractional difference quotients to systematically recover the weak differentiability of $u$ in the vertical direction. Subsequently, by coupling a refined truncation argument with horizontal difference quotient estimates, we successfully bypass the obstructions imposed by the nonlocality and achieve the higher weak differentiability of the full gradient.

The regularity theory for degenerate elliptic equations satisfying $p$-growth conditions in the Heisenberg group $\mathbb{H}^n$ has been extensively investigated. For the quadratic case ($p=2$), Capogna \cite{CL97} established the H\"{o}lder regularity of Euclidean gradients. This was achieved by proving that the horizontal and vertical gradients of weak solutions belong to the local horizontal Sobolev space $HW_{\mathrm{loc}}^{1,2}(\Omega)$ via the method of fractional difference quotients. These results were extended to the more general setting of Carnot groups in \cite{CL99}.

Regarding the sub-elliptic $p$-Laplacian equation
\begin{equation}\label{eq12}
- \sum_{i = 1}^{2n} X_i \left( \left( \mu + |\nabla_H u|^2 \right)^{\frac{p - 2}{2}} X_i u \right) = 0, \quad \mu > 0,
\end{equation}
Domokos \cite{DA04} deduced the weak differentiability of solutions for $1 < p < 4$. Specifically, it was shown that $Tu \in L_{loc}^p(\Omega)$, $u \in HW_{\mathrm{loc}}^{2,p}(\Omega)$(whose definition can be found in Section 2.1) for $\frac{\sqrt{17}-1}{2} \le p \le 2$, and $\nabla_H Tu \in L_{loc}^2(\Omega)$, thereby generalizing the results in \cite{MS03, MS031}. For $p$ sufficiently close to $2$, Manfredi and Mingione \cite{MM07} derived the Lipschitz continuity and $C^\infty$ smoothness of weak solutions to \eqref{eq12} using the approach from \cite{CL97}. Furthermore, Domokos and Manfredi \cite{DM05} employed Calder\'{o}n-Zygmund theory to establish $C^{1,\alpha}$ regularity for weak solutions to both \eqref{eq12} and the degenerate equation
\begin{equation}\label{eq13}
\sum_{i = 1}^{2n} X_i \left( |\nabla_H u|^{p - 2} X_i u \right) = 0.
\end{equation}

Building upon the work in \cite{DA04}, Mingione, Zatorska-Goldstein and Zhong \cite{MZZ07} obtained $C^{1,\alpha}$ regularity for \eqref{eq12} and Lipschitz continuity for \eqref{eq13} in the range $2 \le p < 4$. By employing energy estimates, interpolation inequalities and the double-bootstrap argument, they successfully expanded the integrability of $Tu$ from $p < 4$ to the full range $1 < p < \infty$, effectively breaking the restriction imposed by the homogeneous topological dimension $Q=2n+2$. For $p > 4$ in the specific case of $\mathbb{H}^1$, Ricciotti \cite{R18} provided the local H\"{o}lder regularity of horizontal derivatives (see also \cite{R15} for a comprehensive review). More recently, Domokos and Manfredi \cite{DM20} extended $C^{1,\alpha}$ regularity for $p \ge 2$ to general Lie groups. Mukherjee and Zhong \cite{MZ17} achieved $C^{1,\alpha}$ regularity for the entire range $1 < p < \infty$ by utilizing Moser iteration and fine oscillation estimates.

Beyond standard $p$-growth, Zhang and Niu \cite{ZN20} investigated the $C^\alpha$ regularity of weak solutions to nonlinear equations under generalized Orlicz growth conditions, including $(p,q)$-growth and variable exponent growth, and further established fractional estimates in \cite{ZN22}. Additionally, Zhang and Li \cite{ZL23} proved the weak differentiability of non-uniformly nonlinear degenerate elliptic systems under $(p,q)$-growth conditions.

The $p$-fractional subLaplace operator has emerged as a fundamental operator in diverse fields, including quantum mechanics, image segmentation models, and ferromagnetic analysis. We first recall the linear regime ($p=2$), where the fractional subLaplace operator $(-\Delta_{\mathbb{H}^n})^s$ for $s \in (0,1)$ was initially introduced in the context of conformal invariance by Branson et al. \cite{BFM13}. Within the Heisenberg group $\mathbb{H}^n$, this operator is defined via the spectral multiplier formula:
$$
(-\Delta_{\mathbb{H}^n})^s := 2^s |T|^s \frac{\Gamma \left( -\frac{1}{2} \Delta_{\mathbb{H}^n} |T|^{-1} + \frac{1+s}{2} \right)}{\Gamma \left( -\frac{1}{2} \Delta_{\mathbb{H}^n} |T|^{-1} + \frac{1-s}{2} \right)},
$$
where $\Gamma(\cdot)$ denotes the Euler Gamma function, $T$ is the vertical vector field, and $\Delta_{\mathbb{H}^n}$ is the standard Kohn-Spencer subLaplace operator. Subsequently, Roncal and Thangavelu \cite{RT16} derived that this operator admits the following singular integral representation:
\begin{equation}\label{eq14}
(-\Delta_{\mathbb{H}^n})^s u(x) := C(Q,s) \, \mathrm{P.V.} \int_{\mathbb{H}^n} \frac{u(x) - u(y)}{\| y^{-1} \circ x \|_{\mathbb{H}^n}^{Q + 2s}} \, dy, \quad x \in \mathbb{H}^n,
\end{equation}
 where $Q = 2n+2$ is the homogeneous dimension of $\mathbb{H}^n$ and $C(Q,s) > 0$ is a normalization constant.

Over the past decade, the linear theory surrounding \eqref{eq14} has been extensively developed. Notable contributions include the derivations of Hardy and uncertainty inequalities on stratified Lie groups \cite{CCR15}, the development of Sobolev and Morrey-type embedding theorems for the fractional horizontal Sobolev space $H^s(\mathbb{H}^n)$ \cite{AM18}, and the establishment of Harnack and H\"{o}lder estimates in Carnot groups \cite{FF15}. Furthermore, Liouville-type theorems were explored in \cite{CT16}. For a broader overview of the linear theory, we refer the reader to \cite{FMPPS18, FGMT15, GT22, GT21} and the references therein.

More recently, the attention has shifted toward the nonlinear analogue of \eqref{eq14}, namely the fractional $p$-subLaplace operator ($p \neq 2$). Regarding the regularity theory in this setting, Manfredini et al. \cite{MPP23} proved the interior boundedness and H\"{o}lder continuity of weak solutions by adapting the De Giorgi-Nash-Moser iteration to the nonlocal sub-Riemannian framework. In a complementary study, Palatucci and Piccinini \cite{PP22} established a nonlocal Harnack inequality and characterized the asymptotic behavior of the operator as $s \to 1^-$. Additionally, the study of obstacle problems associated with the nonlocal $p$-subLaplace equation was systematically advanced by Piccinini \cite{P22}, who demonstrated solvability, semicontinuity, and global H\"{o}lder regularity up to the boundary. Further estimates and fundamental functional inequalities in this direction can be found in \cite{PP23, KS18, KS20}. Finally, for the more general class of nonlocal double-phase equations, Fang, Zhang, and Zhang \cite{FZZ24} recently established local regularity results.

Now, we present the definition of weak solutions and state our main regularity theorems. Let us give the notion of a weak solution to the mixed local and nonlocal equation \eqref{eq11} as follows:

\begin{definition}\label{De11}
Let $\Omega \subset \mathbb{H}^n$ be an open set and $f \in L^2(\Omega)$. A function $u \in HW^{1,2}(\mathbb{H}^n)$(whose definition can be found in Section 2.1) is said to be a weak solution to equation \eqref{eq11}, if for every test function $\varphi \in HW_0^{1,2}(\Omega)$, the following integral identity holds:
\begin{equation}\label{eq15}
\int_{\mathbb{H}^n} \nabla_H u \cdot \nabla_H \varphi \, dx + \frac{c_{Q,s}}{2} \iint_{\mathbb{H}^n \times \mathbb{H}^n} \frac{(u(x) - u(y))(\varphi(x) - \varphi(y))}{\| y^{-1} \circ x \|_{\mathbb{H}^n}^{Q + 2s}} \, dx dy = \int_\Omega f \varphi \, dx.
\end{equation}
\end{definition}

Our first main result is $L^2$-estimates for the vertical derivatives of the weak solution:

\begin{theorem}\label{Th11}
Let $u \in HW^{1,2}(\mathbb{H}^n)$ be a weak solution to \eqref{eq11} with $f \in L^2(\Omega)$. Then the vertical derivative $Tu$ belongs to $HW_{\mathrm{loc}}^{1,2}(\Omega)$. Furthermore, for any ball $B(x_0, 2r) \subset \Omega$, there exists a positive constant $c = c(Q,s)$ such that the following estimates hold:
\begin{equation}\label{eq16}
\| Tu \|_{L^2(B(x_0, r/2))} \le c \left( \| u \|_{HW^{1,2}(\mathbb{H}^n)} + \| f \|_{L^2(\Omega)} \right),
\end{equation}
and
\begin{equation}\label{eq17}
\| T \nabla_H u \|_{L^2(B(x_0, r/4))} \le c \left( \| u \|_{HW^{1,2}(\mathbb{H}^n)} + \| f \|_{L^2(\Omega)} \right).
\end{equation}
\end{theorem}

Building upon the estimates for the vertical derivatives, we have the following second-order horizontal regularity:

\begin{theorem}\label{Th12}
Under the assumptions of Theorem \ref{Th11}, the weak solution $u$ possesses second-order horizontal derivatives in the local sense, i.e., $u \in HW_{\mathrm{loc}}^{2,2}(\Omega)$. Moreover, for any $B(x_0, 2r) \subset \Omega$, there exists a constant $c = c(Q,s) > 0$ such that
\begin{equation}\label{eq18}
\| \nabla_H^2 u \|_{L^2(B(x_0, r/4))} \le c \left( \| u \|_{HW^{1,2}(\mathbb{H}^n)} + \| f \|_{L^2(\Omega)} \right).
\end{equation}
\end{theorem}

Finally, by applying a truncation argument with refined difference quotient estimates, we can establish the higher regularity of the weak solution as follows:

\begin{theorem}\label{Th13}
Let $m \in \mathbb{N} \cup \{0\}$. Suppose $f \in HW^{m,2}(\Omega)$ and $u \in HW^{1,2}(\mathbb{H}^n)$ is a weak solution to \eqref{eq11}. Then $u$ belongs to the higher-order horizontal Sobolev space $HW_{\mathrm{loc}}^{m+2,2}(\Omega)$.
\end{theorem}

This paper is organized as follows. In Section 2, we recall the fundamental preliminaries concerning the Heisenberg group $\mathbb{H}^n$, the definition of horizontal Sobolev spaces, and several key lemmas regarding difference quotient estimates. Section 3 is devoted to establishing the weak differentiability of the vertical derivative $Tu$ via fractional difference quotients, thereby proving Theorem \ref{Th11}. Building upon the estimates obtained in Section 3, we further deduce the second-order horizontal regularity and complete the proof of Theorem \ref{Th12} in Section 4. Finally, in Section 5, by combining a truncation argument with the difference quotient method, we establish the higher-order weak differentiability of the gradients and present the complete proof of Theorem \ref{Th13}.

\section{Preliminaries}
\subsection{The Heisenberg group $\mathbb{H}^n$}

The Heisenberg group $\mathbb{H}^n$ is identified with the Euclidean space $\mathbb{R}^{2n+1}$ ($n \ge 1$), equipped with the non-commutative group multiplication defined by
\begin{equation}\label{eq21}
x \circ y = \left( x_1 + y_1, \dots, x_{2n} + y_{2n}, t + \tau + \frac{1}{2}\sum_{i=1}^n (x_i y_{n+i} - x_{n+i} y_i) \right),
\end{equation}
for $x = (x', t) = (x_1, \dots, x_{2n}, t)$ and $y = (y', \tau) = (y_1, \dots, y_{2n}, \tau)$ in $\mathbb{R}^{2n+1}$. The basis for the Lie algebra of left-invariant vector fields on $\mathbb{H}^n$ is given by
\begin{equation}\label{eq22}
X_i = \partial_{x_i} - \frac{x_{n+i}}{2}\partial_t, \quad X_{n+i} = \partial_{x_{n+i}} + \frac{x_i}{2}\partial_t, \quad i = 1, \dots, n.
\end{equation}
The unique non-trivial commutation relation is represented by 
\[ T = \partial_t = [X_i, X_{n+i}] = X_i X_{n+i} - X_{n+i} X_i, \quad i = 1, \dots, n. \]
We refer to $\{X_1, \dots, X_{2n}\}$ as the horizontal vector fields, which span the horizontal distribution, and $T$ as the vertical vector field. For a smooth function $u$ on $\mathbb{H}^n$, the horizontal gradient of $u$ is denoted by $\nabla_H u = (X_1 u, \dots, X_{2n} u)$, and the subLaplace operator is denoted as $\Delta_{\mathbb{H}^n} = \sum_{i=1}^{2n} X_i^2$. Note that the Haar measure on $\mathbb{H}^n$ coincides with the standard Lebesgue measure on $\mathbb{R}^{2n+1}$, and we denote the measure of a set $E \subset \mathbb{H}^n$ by $|E|$.

The intrinsic geometry of $\mathbb{H}^n$ is characterized by the Carnot-Carath\'{e}odory metric (C-C metric), denoted by $d(x, y)$, which represents the infimum of the lengths of all horizontal curves connecting $x$ and $y$. The corresponding metric ball is $B_R(x) = \{y \in \mathbb{H}^n : d(y, x) < R\}$. This metric is globally equivalent to the Kor\'{a}nyi gauge metric $d(x, y) = \|x^{-1} \circ y\|_{\mathbb{H}^n}$, here the homogeneous norm is defined as
\[ \|x\|_{\mathbb{H}^n} = \left( \left( \sum_{i=1}^{2n} x_i^2 \right)^2 + t^2 \right)^{1/4}. \]

Any left-invariant vector field $Z$ can be expressed as a linear combination $Z = \sum_{l=1}^{2n} z_l X_l + z_{2n+1} T$. The exponential map $\exp: \mathfrak{h}^n \to \mathbb{H}^n$ provides a global diffeomorphism. Given $Z$ and $Y$ in the Lie algebra, the Baker-Campbell-Hausdorff formula implies
\begin{equation}\label{eq23}
e^Z e^Y =e^{ Z + Y + \frac{1}{2}[Z, Y] }.
\end{equation}

To analyze the regularity of functions, we introduce the Nirenberg difference operators along a vector field $Z$. For $h \in \mathbb{R} \setminus \{0\}$, the first and second-order differences are defined respectively by
\begin{align}
\Delta_{Z,h} v(x) &= v(x e^{hZ}) - v(x), \label{eq24} \\
\Delta_{Z,h}^2 v(x) &= v(x e^{hZ}) + v(x e^{-hZ}) - 2v(x). \label{eq25}
\end{align}
For $\alpha \in (0, 1]$, we define the fractional difference quotients in the $Z$-direction as
\begin{equation}\label{eq26}
D_{Z,h,\alpha} v(x) = \frac{v(x e^{hZ}) - v(x)}{|h|^\alpha}, \quad D_{Z,-h,\alpha} v(x) = \frac{v(x e^{-hZ}) - v(x)}{-|h|^\alpha}, \quad h > 0.
\end{equation}
Specifically, when $\alpha=1$ and $h \in \mathbb{R}$, we write $v_h(x) = D_{Z,h} v(x) = \frac{v(x e^{hZ}) - v(x)}{h}$. The mixed quadratic difference quotient is given by
\begin{equation}\label{eq28}
D_{Z,-h,\alpha} D_{Z,h,\beta} v(x) = \frac{\Delta_{Z,h}^2 v(x)}{|h|^{\alpha+\beta}}.
\end{equation}
The integration-by-parts formula for these operators holds: for $f, g \in L_{\mathrm{loc}}^1(\Omega)$,
\begin{equation}\label{eq116}
\int_\Omega f(x) D_{Z,h,\alpha} g(x) \, dx = - \int_\Omega D_{Z,-h,\alpha} f(x) g(x) \, dx.
\end{equation}
Further details can be found in \cite{CL97, DA04}.

Finally, we define the relevant functional spaces. For $1 \le p < \infty$ and $k \in \mathbb{N}^+$, the horizontal Sobolev space $HW^{k,p}(\Omega)$ is defined as the set of functions $u \in L^p(\Omega)$ such that all distributional horizontal derivatives up to order $k$ belong to $L^p(\Omega)$. This space is a Banach space equipped with the norm
\[ \|u\|_{HW^{k,p}(\Omega)} = \|u\|_{L^p(\Omega)} + \sum_{m=1}^k \|\nabla_H^m u\|_{L^p(\Omega)}. \]
The local horizontal Sobolev space $HW_{\mathrm{loc}}^{k,p}(\Omega)$ consists of functions $u$ such that $u \in HW^{k,p}(\Omega')$ for every compactly contained subdomain $\Omega' \subset \subset \Omega$. Furthermore, we denote by $HW_0^{k,p}(\Omega)$ the completion of $C_0^\infty(\Omega)$ with respect to the $HW^{k,p}(\Omega)$ norm.

For $s \in (0, 1)$, the fractional horizontal Sobolev space $HW^{s,p}(\mathbb{H}^n)$ is characterized by the Gagliardo-type semi-norm:
\[ [u]_{HW^{s,p}(\mathbb{H}^n)} = \left( \iint_{{\mathbb{H}^n} \times{\mathbb{H}^n}} \frac{|u(x) - u(y)|^p}{\|y^{-1} \circ x\|_{\mathbb{H}^n}^{Q+sp}} \, dx \, dy \right)^{1/p}, \]
where $Q = 2n+2$ is the homogeneous dimension of $\mathbb{H}^n$. The space $HW^{s,p}(\mathbb{H}^n)$ is endowed with the natural norm $\|u\|_{HW^{s,p}} = (\|u\|_{L^p}^p + [u]_{HW^{s,p}}^p)^{1/p}$.

\subsection{Preliminary Lemmas}

We collect several fundamental results concerning finite difference quotients and derivative estimates in the Heisenberg group $\mathbb{H}^n$. Let us begin with a version of the Campbell-Hausdorff formula adapted to left-invariant vector fields.

\begin{lemma}[Campbell-Hausdorff Formula, \cite{MZZ07}, Lemma 2.3]\label{Le21}
Let $\Omega \subset \mathbb{H}^n$ be an open set, and let $Z$ and $Y$ be two left-invariant vector fields. Suppose $v \in HW^{1,p}(\Omega)$ such that $[Y,Z]v \in L_{loc}^p(\Omega)$ for $1 \le p < \infty$. Setting $\tilde{v}(x) := v(x e^{hZ})$, then $Y\tilde{v} \in L_{loc}^p(\Omega)$, and for $x, x e^{hZ} \in \Omega$, we have
\begin{equation}\label{eq29}
Y(v(x e^{hZ})) = Y\tilde{v}(x) = Yv(x e^{hZ}) + h[Y,Z]v(x e^{hZ}).
\end{equation}
Furthermore, for any $h \neq 0$, the following commutation identity holds:
\begin{equation}\label{eq210}
Y(\Delta_{Z,h}v(x)) = \Delta_{Z,h}(Yv)(x) + h[Y,Z]v(x e^{hZ}).
\end{equation}
\end{lemma}

The following lemma characterizes the Sobolev space $HW^{1,p}$ in terms of the boundedness of finite difference quotients.

\begin{lemma}[\cite{CL97}, Proposition 2.3]\label{Le22}
Let $\Omega \subset \mathbb{H}^n$ be an open set and $K \subset \Omega$ be a compact subset. Let $Z$ be a left-invariant vector field and $v \in L_{loc}^p(\Omega)$ for $1 \le p < \infty$. If there exist constants $\tilde{h} > 0$ and $c > 0$ such that
\[
\sup_{0 < |h| < \tilde{h}} \int_K |{D_{Z,h,1}}v(x)|^p \, dx \le c^p,
\]
then the distributional derivative $Zv$ belongs to $L^p(K)$ and $\|Zv\|_{L^p(K)} \le c$. Conversely, if $Zv \in L^p(K)$, then for sufficiently small $h$,
\[
\sup_{0 < |h| < \tilde{h}} \int_K |{D_{Z,h,1}}v(x)|^p \, dx \le (2\|Zv\|_{L^p(K)})^p.
\]
\end{lemma}

For the vertical vector field $T$, we have the following fractional-order estimate in terms of the horizontal gradient.

\begin{lemma}[\cite{DA04}, Proposition 2.2]\label{Le23}
Let $\Omega \subset \mathbb{H}^n$ be an open set and $v \in HW_{\mathrm{loc}}^{1,p}(\Omega)$ for $1 \le p < \infty$. For any $x_0 \in \Omega$ and $r > 0$ such that $B(x_0, 3r) \subset \Omega$, there exists a constant $c > 0$ such that
\[
\int_{B(x_0, r)} |D_{T,h,1/2}v(x)|^p \, dx \le c \int_{B(x_0, 2r)} (|v|^p + |\nabla_H v|^p) \, dx.
\]
\end{lemma}

The following results established in \cite{CL97}, characterize the relationship between fractional vertical derivatives and horizontal gradients, which plays a pivotal role in our regularity bootstrap argument.

Let $\Omega \subset \mathbb{H}^n$ be an open set, and let $Z$ denote a left-invariant vector field on $\mathbb{H}^n$. For a function $\omega \in L^2(\Omega)$ with compact support in $\Omega$ and for any $\alpha \in (0, 1)$, we define the fractional difference seminorm as
\[\left| \omega  \right|_{Z,\alpha }^2 = \mathop {\sup }\limits_{0 < \varepsilon  < {\varepsilon _0}} \mathop {\sup }\limits_{\left| h \right| < \varepsilon } \int_\Omega  {{{\left| h \right|}^{-2\alpha }}{{\left| {\omega \left( {x{e^{hZ}}} \right) - \omega \left( x \right)} \right|}^2}dx}, \]
here ${\varepsilon _0}$ may be chosen sufficiently small.

Let $w \in C_0^\infty(\mathbb{H}^n)$ and $\alpha \in (0,1)$. The $L^2$-norm of the fractional derivative of $w$ along $T$ is given by the following formula:
\begin{equation*}
    \|\partial_t^\alpha w\|_{L^2(\mathbb{H}^n)}^2 = \int_{\mathbb{H}^n} |h|^{2\alpha} |\hat{w}(x,h)|^2 \, dx \, dh,
\end{equation*}
where $\hat{w}$ denotes the partial Fourier transform with respect to the variable $t$.

\begin{lemma}[\cite{CL97}, Theorem 2.5]\label{Le24}
Let $0 < \beta < \alpha < 1$ and $\omega \in C_0^\infty(\mathbb{H}^n)$. There exists a positive constant $c = c(\alpha, \beta)$ such that
\[
c \|\partial_t^\beta \omega\|_{L^2(\mathbb{H}^n)} \le |\omega|_{T,\alpha} \le c^{-1} \|\partial_t^\alpha \omega\|_{L^2(\mathbb{H}^n)}.
\]
\end{lemma}

\begin{lemma}[\cite{CL97}, Theorem 2.6]\label{Le25}
Let $\Omega \subset \mathbb{H}^n$ and $B(x_0, r) \subset \Omega$. For any $\eta \in C_0^\infty(B(x_0, r))$ and $\omega \in C^\infty(\Omega)$, there exists a constant $c > 0$ such that
\[
|\omega \eta|_{T,1/2} \le c \sum_{i = 1}^{2n} |\omega \eta|_{X_i,1}.
\]
\end{lemma}

By choosing a cutoff function $\eta$ such that $\eta = 1$ on $B(x_0, r/2)$ and combining Lemma \ref{Le24} with Lemma \ref{Le25}, we obtain for any $0 < \tau < 1/2$,
\begin{equation}\label{eq211}
\|\partial_t^{1/2 - \tau} \omega\|_{L^2(B(x_0, r/2))} \le c |\omega|_{T,1/2} \le c |\omega \eta|_{T,1/2} \le c \sum_{i = 1}^{2n} |\omega \eta|_{X_i,1} \le c \|\nabla_H(\omega \eta)\|_{L^2(\Omega)}.
\end{equation}
Specifically, for $u \in C^\infty(\Omega)$ and $\phi \in C_0^\infty(B(x_0, r))$, setting $\omega = \partial_t^{1/2 - \tau}(\phi u)$ in \eqref{eq211} yields the following higher estimate:
\begin{equation}\label{eq212}
\|\partial_t^{1 - 2\tau}(\phi u)\|_{L^2(B(x_0, r/2))} \le c \|\nabla_H [(\partial_t^{1/2 - \tau}(\phi u))\eta]\|_{L^2(\Omega)}.
\end{equation}

\section{Proof of Theorem \ref{Th11}}

Let $\eta \in C_0^\infty(\mathbb{H}^n)$ be a standard cutoff function supported in $B(x_0, r)$ such that $\eta \equiv 1$ on $B(x_0, r/2)$, $0 \le \eta \le 1$, and
\begin{equation*}
|\nabla_H \eta| \le c, \quad |T\eta| \le c.
\end{equation*}
By choosing the test function in \eqref{eq15} as
\[\varphi  = {D_{T, - h,\frac{1}{2} }}\left( {{\eta ^2}{D_{T,h,\frac{1}{2} }}u\left( x \right)} \right),\]
and applying the integration-by-parts formula for difference quotients \eqref{eq116}, we obtain
\begin{align}\label{eq31}
&\int_\Omega \langle D_{T, h, 1/2} \nabla_H u, \nabla_H (\eta^2 D_{T, h, 1/2} u) \rangle \, dx \nonumber \\
& \quad + \frac{c_{Q,s}}{2} \iint_{\mathbb{H}^n \times \mathbb{H}^n} \frac{\left( D_{T,h,1/2} u(x) - D_{T,h,1/2} u(y)\right) \cdot \left( \left( {{\eta ^2}{D_{T,h,\frac{1}{2}}}u} \right)\left( x \right) - \left( {{\eta ^2}{D_{T,h,\frac{1}{2}}}u} \right)\left( y \right)\right)}{\|y^{-1} \circ x\|_{\mathbb{H}^n}^{Q+2s}} \, dx \, dy \nonumber \\
&= - \int_\Omega f D_{T, -h, 1/2} (\eta^2 D_{T, h, 1/2} u) \, dx.
\end{align}
Expanding the terms using the discrete Leibniz rule, we have
\[ \nabla_H (\eta^2 D_{T, h, 1/2} u) = \eta^2 D_{T, h, 1/2} \nabla_H u + 2\eta \nabla_H \eta D_{T, h, 1/2} u, \]
and for the kernel part:
\begin{align*}
   & \left( {{\eta ^2}{D_{T,h,\frac{1}{2}}}u} \right)\left( x \right) - \left( {{\eta ^2}{D_{T,h,\frac{1}{2}}}u} \right)\left( y \right) = {\eta ^2}\left( x \right){D_{T,h,\frac{1}{2}}}u\left( x \right) - {\eta ^2}\left( y \right){D_{T,h,\frac{1}{2}}}u\left( y \right) \\
   =& {\eta ^2}\left( x \right)\left( {{D_{T,h,\frac{1}{2}}}u\left( x \right) - {D_{T,h,\frac{1}{2}}}u\left( y \right)} \right) + \left( {{\eta ^2}\left( x \right) - {\eta ^2}\left( y \right)} \right){D_{T,h,\frac{1}{2}}}u\left( y \right)\\
   =& {\eta ^2}\left( x \right)\left( {{D_{T,h,\frac{1}{2}}}u\left( x \right) - {D_{T,h,\frac{1}{2}}}u\left( y \right)} \right) + \eta \left( x \right)\left( {\eta \left( x \right) - \eta \left( y \right)} \right){D_{T,h,\frac{1}{2}}}u\left( y \right)\\
   & +\eta \left( y \right)\left( {\eta \left( x \right) - \eta \left( y \right)} \right){D_{T,h,\frac{1}{2}}}u\left( y \right).
\end{align*}
Substituting these into \eqref{eq31}, we derive
\begin{align}\label{eq32}
  &\int_\Omega  {{\eta ^2}{{\left| {{D_{T,h,\frac{1}{2}}}{\nabla _H}u} \right|}^2}dx}  + \frac{{{c_{Q,s}}}}{2}\iint_{{\mathbb{H}^n} \times {\mathbb{H}^n}} {\frac{{{\eta ^2}\left( x \right){{\left| {{D_{T,h,\frac{1}{2}}}u\left( x \right) - {D_{T,h,\frac{1}{2}}}u\left( y \right)} \right|}^2}}}{{{{\left\| {{y^{ - 1}} \circ x} \right\|}_{\mathbb{H}^n}^{Q + 2s}}}}dxdy}  \nonumber \\
   =&  - 2\int_\Omega  {\eta {\nabla _H}\eta {D_{T,h,\frac{1}{2}}}{\nabla _H}u{D_{T,h,\frac{1}{2}}}udx}  \nonumber\\
   & - \frac{{{c_{Q,s}}}}{2}\iint_{\mathbb{H}^n \times \mathbb{H}^n} {\frac{{\eta \left( x \right)\left( {\eta \left( x \right) - \eta \left( y \right)} \right)\left( {{D_{T,h,\frac{1}{2}}}u\left( x \right) - {D_{T,h,\frac{1}{2}}}u\left( y \right)} \right){D_{T,h,\frac{1}{2}}}u\left( y \right)}}{{{{\left\| {{y^{ - 1}} \circ x} \right\|}_{\mathbb{H}^n}^{Q + 2s}}}}dxdy}\nonumber\\
   & - \frac{{{c_{Q,s}}}}{2}\iint_{\mathbb{H}^n \times \mathbb{H}^n} {\frac{{\eta \left( y \right)\left( {\eta \left( x \right) - \eta \left( y \right)} \right)\left( {{D_{T,h,\frac{1}{2}}}u\left( x \right) - {D_{T,h,\frac{1}{2}}}u\left( y \right)} \right){D_{T,h,\frac{1}{2}}}u\left( y \right)}}{{{{\left\| {{y^{ - 1}} \circ x} \right\|}_{\mathbb{H}^n}^{Q + 2s}}}}dxdy} \nonumber\\
   & - \int_\Omega  {f{D_{T, - h,\frac{1}{2}}}\left( {{\eta ^2}{D_{T,h,\frac{1}{2}}}u} \right)dx}\nonumber\\
    = :&{I_1} + {I_2} + {I_3} + {I_4}.
\end{align}

We now estimate each term $I_i$ ($i=1,\dots,4$) using Young's inequality with $\varepsilon > 0$. For $I_1$, we have
\begin{align}\label{eq33}
   {I_1} & \le 2\int_\Omega  {\eta \left| {{\nabla _H}\eta } \right|\left| {{D_{T,h,\frac{1}{2}}}{\nabla _H}u} \right|\left| {{D_{T,h,\frac{1}{2}}}u} \right|dx}  \nonumber\\
  &  \le \frac{1}{2}\int_\Omega  {{\eta ^2}{{\left| {{D_{T,h,\frac{1}{2}}}{\nabla _H}u} \right|}^2}dx}  + 2\int_\Omega  {{{\left| {{\nabla _H}\eta } \right|}^2}{{\left| {{D_{T,h,\frac{1}{2}}}u} \right|}^2}dx}  .
\end{align}
For the integral terms $I_2$ and $I_3$, a similar application of Young's inequality yields
\begin{align}\label{eq34}
  {I_2} &\le\frac{{{c_{Q,s}}}}{2}\iint_{\mathbb{H}^n \times \mathbb{H}^n} {\frac{{\eta \left( x \right)\left| {\eta \left( x \right) - \eta \left( y \right)} \right|\left| {{D_{T,h,\frac{1}{2}}}u\left( x \right) - {D_{T,h,\frac{1}{2}}}u\left( y \right)} \right|\left| {{D_{T,h,\frac{1}{2}}}u\left( y \right)} \right|}}{{{{\left\| {{y^{ - 1}} \circ x} \right\|}_{\mathbb{H}^n}^{Q + 2s}}}}dxdy}\nonumber\\
    &\le \frac{1}{4}\frac{{{c_{Q,s}}}}{2}\iint_{\mathbb{H}^n \times \mathbb{H}^n} {\frac{{{\eta ^2}\left( x \right){{\left| {{D_{T,h,\frac{1}{2}}}u\left( x \right) - {D_{T,h,\frac{1}{2}}}u\left( y \right)} \right|}^2}}}{{{{\left\| {{y^{ - 1}} \circ x} \right\|}_{\mathbb{H}^n}^{Q + 2s}}}}dxdy} \nonumber\\
     &+ \frac{{{c_{Q,s}}}}{2}\iint_{\mathbb{H}^n \times \mathbb{H}^n} {\frac{{{{\left| {\eta \left( x \right) - \eta \left( y \right)} \right|}^2}{{\left| {{D_{T,h,\frac{1}{2}}}u\left( y \right)} \right|}^2}}}{{{{\left\| {{y^{ - 1}} \circ x} \right\|}_{\mathbb{H}^n}^{Q + 2s}}}}dxdy};
\end{align}
\begin{align}\label{eq35}
  {I_3}     &\le \frac{1}{4}\frac{{{c_{Q,s}}}}{2}\iint_{\mathbb{H}^n \times \mathbb{H}^n} {\frac{{{\eta ^2}\left( y \right){{\left| {{D_{T,h,\frac{1}{2}}}u\left( x \right) - {D_{T,h,\frac{1}{2}}}u\left( y \right)} \right|}^2}}}{{{{\left\| {{y^{ - 1}} \circ x} \right\|}_{\mathbb{H}^n}^{Q + 2s}}}}dxdy} \nonumber\\
     &+ \frac{{{c_{Q,s}}}}{2}\iint_{\mathbb{H}^n \times \mathbb{H}^n} {\frac{{{{\left| {\eta \left( x \right) - \eta \left( y \right)} \right|}^2}{{\left| {{D_{T,h,\frac{1}{2}}}u\left( y \right)} \right|}^2}}}{{{{\left\| {{y^{ - 1}} \circ x} \right\|}_{\mathbb{H}^n}^{Q + 2s}}}}dxdy}.
\end{align}
The term $I_4$ is bounded by
\begin{equation}\label{eq36}
  {I_4}   \le\int_\Omega  {\left| {f{D_{T, - h,\frac{1}{2}}}\left( {{\eta ^2}{D_{T,h,\frac{1}{2}}}u} \right)} \right|dx}  \le\frac{1}{2}c(\varepsilon)\int_\Omega  {{{\left| f \right|}^2}dx}  +\frac{1}{2}\varepsilon \int_\Omega  {{{\left| {{D_{T, - h,\frac{1}{2}}}\left( {{\eta ^2}{D_{T,h,\frac{1}{2}}}u} \right)} \right|}^2}dx} .
\end{equation}
Since
$$\iint_{\mathbb{H}^n \times \mathbb{H}^n} {\frac{{{\eta ^2}\left( x \right){{\left| {{D_{T,h,\frac{1}{2}}}u\left( x \right) - {D_{T,h,\frac{1}{2}}}u\left( y \right)} \right|}^2}}}{{{{\left\| {{y^{ - 1}} \circ x} \right\|}_{\mathbb{H}^n}^{Q + 2s}}}}dxdy}=\iint_{\mathbb{H}^n \times \mathbb{H}^n} {\frac{{{\eta ^2}\left( y \right){{\left| {{D_{T,h,\frac{1}{2}}}u\left( x \right) - {D_{T,h,\frac{1}{2}}}u\left( y \right)} \right|}^2}}}{{{{\left\| {{y^{ - 1}} \circ x} \right\|}_{\mathbb{H}^n}^{Q + 2s}}}}dxdy},$$
we have by gathering together estimates \eqref{eq32}-\eqref{eq36} that 
\begin{align*}
   & \int_\Omega  {{\eta ^2}{{\left| {{D_{T,h,\frac{1}{2}}}{\nabla _H}u} \right|}^2}dx}  + \frac{{{c_{Q,s}}}}{2}\iint_{\mathbb{H}^n \times \mathbb{H}^n} {\frac{{{\eta ^2}\left( x \right){{\left| {{D_{T,h,\frac{1}{2}}}u\left( x \right) - {D_{T,h,\frac{1}{2}}}u\left( y \right)} \right|}^2}}}{{{{\left\| {{y^{ - 1}} \circ x} \right\|}_{\mathbb{H}^n}^{Q + 2s}}}}dxdy}  \nonumber\\
   \le& \frac{1}{2}\int_\Omega  {{\eta ^2}{{\left| {{D_{T,h,\frac{1}{2}}}{\nabla _H}u} \right|}^2}dx} +2\int_\Omega  {{{\left| {{\nabla _H}\eta } \right|}^2}{{\left| {{D_{T,h,\frac{1}{2}}}u} \right|}^2}dx}\\
  & +\frac{1}{2}\frac{{{c_{Q,s}}}}{2}\iint_{\mathbb{H}^n \times \mathbb{H}^n} {\frac{{{\eta ^2}\left( x \right){{\left| {{D_{T,h,\frac{1}{2}}}u\left( x \right) - {D_{T,h,\frac{1}{2}}}u\left( y \right)} \right|}^2}}}{{{{\left\| {{y^{ - 1}} \circ x} \right\|}_{\mathbb{H}^n}^{Q + 2s}}}}dxdy}\\
   &+{{c_{Q,s}}}\iint_{\mathbb{H}^n \times \mathbb{H}^n} {\frac{{{{\left| {\eta \left( x \right) - \eta \left( y \right)} \right|}^2}{{\left| {{D_{T,h,\frac{1}{2}}}u\left( y \right)} \right|}^2}}}{{{{\left\| {{y^{ - 1}} \circ x} \right\|}_{\mathbb{H}^n}^{Q + 2s}}}}dxdy}\nonumber\\
    &+\frac{1}{2}c(\varepsilon)\int_\Omega  {{{\left| f \right|}^2}dx} + \frac{1}{2}\varepsilon\int_\Omega  {{{\left| {{D_{T, - h,\frac{1}{2}}}\left( {{\eta ^2}{D_{T,h,\frac{1}{2}}}u} \right)} \right|}^2}dx}\nonumber\\
    =&\frac{1}{2}\left( \int_\Omega  {{\eta ^2}{{\left| {{D_{T,h,\frac{1}{2}}}{\nabla _H}u} \right|}^2}dx}  + \frac{{{c_{Q,s}}}}{2}\iint_{\mathbb{H}^n \times \mathbb{H}^n} {\frac{{{\eta ^2}\left( x \right){{\left| {{D_{T,h,\frac{1}{2}}}u\left( x \right) - {D_{T,h,\frac{1}{2}}}u\left( y \right)} \right|}^2}}}{{{{\left\| {{y^{ - 1}} \circ x} \right\|}_{\mathbb{H}^n}^{Q + 2s}}}}dxdy} \right)\nonumber\\
    &+2\int_\Omega  {{{\left| {{\nabla _H}\eta } \right|}^2}{{\left| {{D_{T,h,\frac{1}{2}}}u} \right|}^2}dx}+{{c_{Q,s}}}\iint_{\mathbb{H}^n \times \mathbb{H}^n} {\frac{{{{\left| {\eta \left( x \right) - \eta \left( y \right)} \right|}^2}{{\left| {{D_{T,h,\frac{1}{2}}}u\left( y \right)} \right|}^2}}}{{{{\left\| {{y^{ - 1}} \circ x} \right\|}_{\mathbb{H}^n}^{Q + 2s}}}}dxdy}\nonumber\\
    &+\frac{1}{2}c(\varepsilon)\int_\Omega  {{{\left| f \right|}^2}dx} + \frac{1}{2}\varepsilon\int_\Omega  {{{\left| {{D_{T, - h,\frac{1}{2}}}\left( {{\eta ^2}{D_{T,h,\frac{1}{2}}}u} \right)} \right|}^2}dx}.
\end{align*}
That is
\begin{align}\label{eq37}
   & \int_\Omega  {{\eta ^2}{{\left| {{D_{T,h,\frac{1}{2}}}{\nabla _H}u} \right|}^2}dx}  + \frac{{{c_{Q,s}}}}{2}\iint_{\mathbb{H}^n \times \mathbb{H}^n} {\frac{{{\eta ^2}\left( x \right){{\left| {{D_{T,h,\frac{1}{2}}}u\left( x \right) - {D_{T,h,\frac{1}{2}}}u\left( y \right)} \right|}^2}}}{{{{\left\| {{y^{ - 1}} \circ x} \right\|}_{\mathbb{H}^n}^{Q + 2s}}}}dxdy}  \nonumber\\
   \le& 4\int_\Omega  {{{\left| {{\nabla _H}\eta } \right|}^2}{{\left| {{D_{T,h,\frac{1}{2}}}u} \right|}^2}dx} + 2{{c_{Q,s}}}\iint_{\mathbb{H}^n \times \mathbb{H}^n} {\frac{{{{\left| {\eta \left( x \right) - \eta \left( y \right)} \right|}^2}{{\left| {{D_{T,h,\frac{1}{2}}}u\left( y \right)} \right|}^2}}}{{{{\left\| {{y^{ - 1}} \circ x} \right\|}_{\mathbb{H}^n}^{Q + 2s}}}}dxdy}\nonumber\\
    &+c(\varepsilon)\int_\Omega  {{{\left| f \right|}^2}dx} + \varepsilon\int_\Omega  {{{\left| {{D_{T, - h,\frac{1}{2}}}\left( {{\eta ^2}{D_{T,h,\frac{1}{2}}}u} \right)} \right|}^2}dx}.
\end{align}

On the one hand, since $\eta  \in C_0^\infty \left( {{\mathbb{H}^n}} \right) $, it holds that for every $y\in {\mathbb{H}^n}$,
\begin{align*}
   \int_{\mathbb{H}^n } {\frac{{{{\left| {\eta \left( x \right) - \eta \left( y \right)} \right|}^2}}}{{{{\left\| {{y^{ - 1}} \circ x} \right\|}_{\mathbb{H}^n}^{Q + 2s}}}}dx}
     \le&\mathop {\sup }\limits_{{\mathbb{H}^n}} {\left| {{\nabla _H}\eta } \right|^2}\int_{\left\{ {{{\left\| {{y^{ - 1}} \circ x} \right\|}_{\mathbb{H}^n}} \le1 } \right\}}
     {\frac{{dx}}{{{{\left\| {{y^{ - 1}} \circ x} \right\|}_{\mathbb{H}^n}^{Q + 2(s-1)} }}}}\nonumber \\
   & +4\mathop {\sup }\limits_{{\mathbb{H}^n}} {\left| {\eta } \right|^2}\int_{\left\{ {{{\left\| {{y^{ - 1}} \circ x} \right\|}_{\mathbb{H}^n}} >1 } \right\}} {\frac{{dx}}{{{{\left\| {{y^{ - 1}} \circ x} \right\|}_{\mathbb{H}^n}^{Q + 2s}}}}} \nonumber \\
=&:c(Q,\eta,s).
\end{align*}
Then
\begin{equation}\label{eq38}
\iint_{\mathbb{H}^n \times \mathbb{H}^n} {\frac{{{{\left| {\eta \left( x \right) - \eta \left( y \right)} \right|}^2}{{\left| {{D_{T,h,\frac{1}{2}}}u\left( y \right)} \right|}^2}}}{{{{\left\| {{y^{ - 1}} \circ x} \right\|}_{\mathbb{H}^n}^{Q + 2s}}}}dxdy}\le c(Q,\eta,s)\int_{{\mathbb{H}^n}} {{{{{\left| {{D_{T,h,\frac{1}{2}}}u\left( y \right)} \right|}^2}}}dy}.
\end{equation}
On the other hand, we obtain from Lemma \ref{Le23} and the properties of $\eta$ that
\begin{align}\label{eq39}
   \int_\Omega  {{{\left| {{D_{T, - h,\frac{1}{2}}}\left( {{\eta ^2}{D_{T,h,\frac{1}{2}}}u} \right)} \right|}^2}dx}  \le& c\int_\Omega  {\left( {{{\left| {{\eta ^2}{D_{T,h,\frac{1}{2}}}u} \right|}^2} + {{\left| {{\nabla _H}\left( {{\eta ^2}{D_{T,h,\frac{1}{2}}}u} \right)} \right|}^2}} \right)dx}  \nonumber \\
   \le& c\int_\Omega  {\left( {{{\left| {{\eta ^2}{D_{T,h,\frac{1}{2}}}u} \right|}^2} + 4{{\left| {\eta {\nabla _H}\eta {D_{T,h,\frac{1}{2}}}u} \right|}^2} + 2{{\left| {{\eta ^2}\left( {{D_{T,h,\frac{1}{2}}}{\nabla _H}u} \right)} \right|}^2}} \right)dx}  \nonumber\\
  \le&  c\int_{{\mathbb{H}^n}} {{{\left| {{D_{T,h,\frac{1}{2}}}u} \right|}^2}dx}  + c\int_\Omega  {{\eta ^2}{{\left| {{D_{T,h,\frac{1}{2}}}{\nabla _H}u} \right|}^2}dx} .
\end{align}
 Moreover, because of
$$\frac{{{c_{Q,s}}}}{2}\iint_{\mathbb{H}^n \times \mathbb{H}^n} {\frac{{{\eta ^2}\left( x \right){{\left| {{D_{T,h,\frac{1}{2}}}u\left( x \right) - {D_{T,h,\frac{1}{2}}}u\left( y \right)} \right|}^2}}}{{{{\left\| {{y^{ - 1}} \circ x} \right\|}_{\mathbb{H}^n}^{Q + 2s}}}}dxdy}>0,$$
we deduce by combining \eqref{eq37}-\eqref{eq39} and using Lemma \ref{Le23} and the properties of $\eta$ that
\begin{align*}
   & \int_\Omega  {{\eta ^2}{{\left| {{D_{T,h,\frac{1}{2}}}{\nabla _H}u} \right|}^2}dx}  \nonumber\\
  \le& \int_\Omega  {{\eta ^2}{{\left| {{D_{T,h,\frac{1}{2}}}{\nabla _H}u} \right|}^2}dx}  + \frac{{{c_{Q,s}}}}{2}\iint_{\mathbb{H}^n \times \mathbb{H}^n} {\frac{{{\eta ^2}\left( x \right){{\left| {{D_{T,h,\frac{1}{2}}}u\left( x \right) - {D_{T,h,\frac{1}{2}}}u\left( y \right)} \right|}^2}}}{{{{\left\| {{y^{ - 1}} \circ x} \right\|}_{\mathbb{H}^n}^{Q + 2s}}}}dxdy}  \nonumber\\
  \le& 4\int_\Omega  {{{\left| {{\nabla _H}\eta } \right|}^2}{{\left| {{D_{T,h,\frac{1}{2}}}u} \right|}^2}dx} + 2{{c_{Q,s}}}\iint_{\mathbb{H}^n \times \mathbb{H}^n} {\frac{{{{\left| {\eta \left( x \right) - \eta \left( y \right)} \right|}^2}{{\left| {{D_{T,h,\frac{1}{2}}}u\left( y \right)} \right|}^2}}}{{{{\left\| {{y^{ - 1}} \circ x} \right\|}_{\mathbb{H}^n}^{Q + 2s}}}}dxdy}\nonumber\\
    &+c(\varepsilon)\int_\Omega  {{{\left| f \right|}^2}dx} + \varepsilon\int_\Omega  {{{\left| {{D_{T, - h,\frac{1}{2}}}\left( {{\eta ^2}{D_{T,h,\frac{1}{2}}}u} \right)} \right|}^2}dx}\nonumber\\
   \le& 4\int_\Omega  {{{\left| {{\nabla _H}\eta } \right|}^2}{{\left| {{D_{T,h,\frac{1}{2}}}u} \right|}^2}dx} + {{c{(Q,\eta,s,\varepsilon)}}}\int_{{\mathbb{H}^n}} {{{{{\left| {{D_{T,h,\frac{1}{2}}}u} \right|}^2}}}dx}+c(\varepsilon)\int_\Omega  {{{\left| f \right|}^2}dx}\nonumber\\
    &+c\varepsilon\int_\Omega  {{\eta ^2}{{\left| {{D_{T,h,\frac{1}{2}}}{\nabla _H}u} \right|}^2}dx} \nonumber\\
   \le& {{c{(Q,\eta,s,\varepsilon)}}}\int_{{\mathbb{H}^n}} {{{{{\left| {{D_{T,h,\frac{1}{2}}}u} \right|}^2}}}dx}+c(\varepsilon)\int_\Omega  {{{\left| f \right|}^2}dx}+c\varepsilon\int_\Omega  {{\eta ^2}{{\left| {{D_{T,h,\frac{1}{2}}}{\nabla _H}u} \right|}^2}dx}\nonumber\\
    \le & c(Q,\eta ,s,\varepsilon)\left\| u \right\|_{H{W^{1,2}}\left( {{\mathbb{H}^n}} \right)}^2 + c\int_\Omega  {{{\left| f \right|}^2}dx}+ c\varepsilon\int_\Omega  {{\eta ^2}{{\left| {{D_{T,h,\frac{1}{2}}}{\nabla _H}u} \right|}^2}dx}.
\end{align*}
By selecting $\varepsilon$ sufficiently small, we conclude
\begin{equation}\label{eq310}
    \int_\Omega  {{\eta ^2}{{\left| {{D_{T,h,\frac{1}{2}}}{\nabla _H}u} \right|}^2}dx}
   \le  c(Q,\eta ,s)\left\| u \right\|_{H{W^{1,2}}\left( {{\mathbb{H}^n}} \right)}^2 + c\int_\Omega  {{{\left| f \right|}^2}dx} .
\end{equation}
Then it follows by using the properties of $\eta$ and \eqref{eq310} that
\begin{align}\label{eq311}
   \int_\Omega  {{{\left| {{D_{T,h,\frac{1}{2}}}\left( {\eta {\nabla _H}u} \right)} \right|}^2}dx} & \le c\int_\Omega  {{\eta ^2}{{\left| {{D_{T,h,\frac{1}{2}}}{\nabla _H}u} \right|}^2}dx}  + c\int_\Omega  {{{\left| {{D_{T,h,\frac{1}{2}}}\eta } \right|}^2}{{\left| {{\nabla _H}u\left( {x{e^{hT}}} \right)} \right|}^2}dx} \nonumber \\
   &  \le c\int_\Omega  {{\eta ^2}{{\left| {{D_{T,h,\frac{1}{2}}}{\nabla _H}u} \right|}^2}dx}  + c\left( \eta  \right)\int_{{\mathbb{H}^n}} {{{\left| {{\nabla _H}u} \right|}^2}dx} \nonumber\\
    &  \le c(Q,\eta ,s)\left\| u \right\|_{H{W^{1,2}}\left( {{\mathbb{H}^n}} \right)}^2 + c\int_\Omega  {{{\left| f \right|}^2}dx}.
\end{align}
Hence
\begin{equation}\label{eq312}
 \left| {\eta {\nabla _H}u} \right|_{T,{\raise0.5ex\hbox{$\scriptstyle 1$}
\kern-0.1em/\kern-0.15em
\lower0.25ex\hbox{$\scriptstyle 2$}}}^2 \le c(Q,\eta ,s)\left\| u \right\|_{H{W^{1,2}}\left( {{\mathbb{H}^n}} \right)}^2 + c\int_\Omega  {{{\left| f \right|}^2}dx}.
\end{equation}
Then it follows by using Lemma \ref{Le24} and \eqref{eq312} that
\begin{equation}\label{eq313}
  \left\| {\partial _t^{\frac{1}{2} - \epsilon }\left( {\eta {\nabla _H}u} \right)} \right\|_{{L^2}\left( {{\mathbb{H}^n}} \right)}^2 \le c\left| {\left( {\eta {\nabla _H}u} \right)} \right|_{T,\frac{1}{2}}^2 \le  c(Q,\eta ,s)\left\| u \right\|_{H{W^{1,2}}\left( {{\mathbb{H}^n}} \right)}^2 + c\int_\Omega  {{{\left| f \right|}^2}dx} ,
\end{equation}
where $0<\epsilon<\frac{1}{2}$. By \eqref{eq313} and \eqref{eq211}, we have
\begin{align}\label{eq314}
  \left\| {{\nabla _H}\left[ {\partial _t^{\frac{1}{2} - \epsilon }\left( {\eta u} \right)} \right]} \right\|_{{L^2}\left( {{\mathbb{H}^n}} \right)}^2 & = \left\| {\partial _t^{\frac{1}{2} - \epsilon }\left( {{\nabla _H}\left( {\eta u} \right)} \right)} \right\|_{{L^2}\left( {{\mathbb{H}^n}} \right)}^2 \nonumber \\
 & \le \left\| {\partial _t^{\frac{1}{2} - \epsilon }\left( {\eta {\nabla _H}u} \right)} \right\|_{{L^2}\left( {{\mathbb{H}^n}} \right)}^2 + \left\| {\partial _t^{\frac{1}{2} - \epsilon }\left( {u{\nabla _H}\eta } \right)} \right\|_{{L^2}\left( {{\mathbb{H}^n}} \right)}^2 \nonumber \\
  & \le c(Q,\eta ,s)\left\| u \right\|_{H{W^{1,2}}\left( {{\mathbb{H}^n}} \right)}^2+ c\int_\Omega  {{{\left| f \right|}^2}dx}+ c\left\| {{\nabla _H}\left( {u{\nabla _H}\eta } \right)} \right\|_{{L^2}\left( {{\mathbb{H}^n}} \right)}^2 \nonumber \\
  & \le c(Q,\eta ,s)\left\| u \right\|_{H{W^{1,2}}\left( {{\mathbb{H}^n}} \right)}^2+ c\int_\Omega  {{{\left| f \right|}^2}dx},
\end{align}
so we use \eqref{eq212} and \eqref{eq314} to deduce
\begin{align}\label{eq315}
  \left\| {\partial _t^{1 - 2\epsilon }\left( {\eta u} \right)} \right\|_{{L^2}\left( {B\left( {{x_0},r} \right)} \right)}^2 & \le c\left\| {{\nabla _H}\left[ {\partial _t^{\frac{1}{2} - \epsilon }\left( {\eta u} \right)} \right]} \right\|_{{L^2}\left( {{\mathbb{H}^n}} \right)}^2  \nonumber \\
  &\le c(Q,\eta ,s)\left\| u \right\|_{H{W^{1,2}}\left( {{\mathbb{H}^n}} \right)}^2 + c\int_\Omega  {{{\left| f \right|}^2}dx}.
\end{align}
By choosing $\epsilon  = {\epsilon _0} = \frac{1}{6}$ and $\epsilon  =\frac{1}{2} {\epsilon _0} = \frac{1}{12}$ in \eqref{eq315}, respectively, we establish the upper bounds for
$\left\| {\partial _t^{\frac{2}{3} }\left( {\eta u} \right)} \right\|_{{L^2}\left( {B\left( {{x_0},r} \right)} \right)}^2$ and $\left\| {\partial _t^{\frac{5}{6} }\left( {\eta u} \right)} \right\|_{{L^2}\left( {B\left( {{x_0},r} \right)} \right)}^2$.

Now we repeat the argument in \eqref{eq31}-\eqref{eq312} by taking $\varphi  = {D_{T, - h,\frac{5}{6} }}\left( {{\eta ^2}{D_{T,h,\frac{5}{6} }}u\left( x \right)} \right)$ in \eqref{eq15}, where $\eta  \in C_0^\infty \left( {B\left( {{x_0},r} \right)} \right)$ is a cut-off function between $B\left( {{x_0},\frac{r}{2}} \right)$ and $B\left( {{x_0},r} \right)$, and obtain
\begin{equation}\label{eq316}
  \left| {\left( {\eta {\nabla _H}u} \right)} \right|_{T,\frac{5}{6}}^2  \le c(Q,\eta ,s)\left\| u \right\|_{H{W^{1,2}}\left( {{\mathbb{H}^n}} \right)}^2 + c\int_\Omega  {{{\left| f \right|}^2}dx}.
\end{equation}
By Lemma \ref{Le24}, it shows
\begin{equation}\label{eq317}
\left\| {\partial _t^{\frac{5}{6} - {\epsilon _0}}\left( {\eta {\nabla _H}u} \right)} \right\|_{{L^2}\left( {{\mathbb{H}^n}} \right)}^2 \le  \left| {\left( {\eta {\nabla _H}u} \right)} \right|_{T,\frac{5}{6}}^2 \le c(Q,\eta ,s)\left\| u \right\|_{H{W^{1,2}}\left( {{\mathbb{H}^n}} \right)}^2 + c\int_\Omega  {{{\left| f \right|}^2}dx},
\end{equation}
where $0<\epsilon_0 <\frac{1}{2}$, so
\begin{equation}\label{eq318}
  \left\| {{\nabla _H}\left[ {\partial _t^{\frac{5}{6} - {\epsilon _0}}\left( {\eta u} \right)} \right]} \right\|_{{L^2}\left( {{\mathbb{H}^n}} \right)}^2 \le c(Q,\eta ,s)\left\| u \right\|_{H{W^{1,2}}\left( {{\mathbb{H}^n}} \right)}^2 + c\int_\Omega  {{{\left| f \right|}^2}dx} .
\end{equation}
It implies
\begin{align}\label{eq319}
 \left\| {{\partial _t}u} \right\|_{{L^2}\left( {B\left( {{x_0},\frac{r}{2}} \right)} \right)}^2 &\le \left\| {\partial _t^{\frac{1}{2} - {\epsilon _0}}\partial _t^{\frac{5}{6} - {\epsilon _0}}\left( {\eta u} \right)} \right\|_{{L^2}\left( {{\mathbb{H}^n}} \right)}^2 \le c\left\| {{\nabla _H}\left[ {\partial _t^{\frac{5}{6} - {\epsilon _0}}\left( {\eta u} \right)} \right]} \right\|_{{L^2}\left( {{\mathbb{H}^n}} \right)}^2\nonumber\\
 &\le c(Q,\eta ,s)\left\| u \right\|_{H{W^{1,2}}\left( {{\mathbb{H}^n}} \right)}^2 + c\int_\Omega  {{{\left| f \right|}^2}dx}.
\end{align}
Hence
\[Tu \in L_{loc}^2\left( \Omega  \right).\]

Finally, we choose $\varphi  = {D_{T, - h,1 }}\left( {{\eta ^2}{D_{T,h,1 }}u\left( x \right)} \right)$ in \eqref{eq15}, where $\eta  \in C_0^\infty \left( {B\left( {{x_0},\frac{r}{2}} \right)} \right)$ is a cut-off function between $B\left( {{x_0},\frac{r}{4}} \right)$ and $B\left( {{x_0},\frac{r}{2}} \right)$, and have
\[\left| {\left( {\eta {\nabla _H}u} \right)} \right|_{T,1}^2 \le c(Q,\eta ,s)\left\| u \right\|_{H{W^{1,2}}\left( {{\mathbb{H}^n}} \right)}^2 + c\int_\Omega  {{{\left| f \right|}^2}dx}.\]
That is
\[Tu \in HW_{\mathrm{loc}}^{1,2}\left( \Omega  \right).\]

\section{Proof of Theorem \ref{Th12}}

Let
${i_0} \in \left\{ {1, \cdots ,n} \right\},\;0<h <r$. By taking the test function in \eqref{eq15} as
$$\varphi  = {D_{{X_{{i_0}}}, - h,1}}{D_{{X_{{i_0}}},h,1}}\left( {{\eta ^4}u} \right)\left( x \right),$$
where $\eta \in C_0^\infty \left( {\mathbb{H}^n}  \right) $ is a cut-off function between $B\left( {{x_0},\frac{r}{4}} \right)$ and $B\left( {{x_0},\frac{r}{2}} \right)$ with $\left| {{\nabla _H}\eta } \right| \le c$ and $\left| {T\eta } \right| \le c$, we obtain
\begin{align}\label{eq41}
   & \int_\Omega  {\langle {\nabla _H}u,{\nabla _H}\left( {{D_{{X_{{i_0}}}, - h,1}}{D_{{X_{{i_0}}},h,1}}\left( {{\eta ^4}u} \right)} \right)\rangle dx} \nonumber \\
& + \frac{{{c_{Q,s}}}}{2}\iint_{\mathbb{H}^n \times \mathbb{H}^n} {\frac{{\left( {u\left( x \right) - u\left( y \right)} \right)\left( {{D_{{X_{{i_0}}}, - h,1}}{D_{{X_{{i_0}}},h,1}}\left( {{\eta ^4}u} \right)\left( x \right) - {D_{{X_{{i_0}}}, - h,1}}{D_{{X_{{i_0}}},h,1}}\left( {{\eta ^4}u} \right)\left( y \right)} \right)}}{{\left\| {{y^{ - 1}}\circ x} \right\|_{{\mathbb{H}^n}}^{Q + 2s}}}dxdy} \nonumber \\
    =& \int_\Omega  {f{D_{{X_{{i_0}}}, - h,1}}{D_{{X_{{i_0}}},h,1}}\left( {{\eta ^4}u} \right)dx}.
\end{align}

For $i \ne n + {i_0}$, it holds from the commutativity of ${X_i}$, ${D_{{X_{{i_0}}}, - h,1}}$ and ${D_{{X_{{i_0}}},h,1}}$ that
\[{X_i}\left( {{D_{{X_{{i_0}}}, - h,1}}{D_{{X_{{i_0}}},h,1}}\left( {{\eta ^4}u} \right)\left( x \right)} \right) = {D_{{X_{{i_0}}}, - h,1}}{D_{{X_{{i_0}}},h,1}}\left( {{X_i}\left( {{\eta ^4}u} \right)\left( x \right)} \right).\]
For $i = n + {i_0}$, the non-commutative structure necessitates a more delicate treatment. By invoking Lemma \ref{Le21}, we infer that
\begin{align*}
  &{X_{n + {i_0}}}\left( {{D_{{X_{{i_0}}}, - h,1}}{D_{{X_{{i_0}}},h,1}}\left( {{\eta ^4}u} \right)\left( x \right)} \right)\\
 =& {D_{{X_{{i_0}}}, - h,1}}{X_{n + {i_0}}}\left( {{D_{{X_{{i_0}}},h,1}}\left( {{\eta ^4}u} \right)\left( x \right)} \right) - \left[ {{X_{n + {i_0}}},{X_{{i_0}}}} \right]{D_{{X_{{i_0}}},h,1}}\left( {\left( {{\eta ^4}u} \right)\left( {x{e^{ - s{X_{{i_0}}}}}} \right)} \right)\\
= &{D_{{X_{{i_0}}}, - h,1}}{D_{{X_{{i_0}}},h,1}}\left( {{X_{n + {i_0}}}\left( {{\eta ^4}u} \right)\left( x \right)} \right)\\
& - \left[ {{D_{{X_{{i_0}}}, - h,1}}\left( {T\left( {{\eta ^4}u} \right)\left( {x{e^{s{X_{{i_0}}}}}} \right)} \right) - T{D_{{X_{{i_0}}},h,1}}\left( {\left( {{\eta ^4}u} \right)\left( {x{e^{ - s{X_{{i_0}}}}}} \right)} \right)} \right]\\
 = &{D_{{X_{{i_0}}}, - h,1}}{D_{{X_{{i_0}}},h,1}}\left( {{X_{n + {i_0}}}\left( {{\eta ^4}u} \right)\left( x \right)} \right)\\
 &- \left[ {{D_{{X_{{i_0}}},h,1}}\left( {T\left( {{\eta ^4}u} \right)\left( x \right)} \right) + {D_{{X_{{i_0}}}, - h,1}}\left( {T\left( {{\eta ^4}u} \right)\left( x \right)} \right)} \right],
\end{align*}
so
\begin{align*}
   & \int_\Omega  {\langle {\nabla _H}u,{\nabla _H}\left( {{D_{{X_{{i_0}}}, - h,1}}{D_{{X_{{i_0}}},h,1}}\left( {{\eta ^4}u} \right)} \right)\rangle dx}  \\
   =&  \int_\Omega  {\langle {\nabla _H}u,{D_{{X_{{i_0}}}, - h,1}}{D_{{X_{{i_0}}},h,1}}\left( {{\nabla _H}\left( {{\eta ^4}u} \right)} \right)\rangle dx} \\
   & - \int_\Omega  {\langle {X _{n+i_0}}u,\left[ {{D_{{X_{{i_0}}},h,1}}\left( {T\left( {{\eta ^4}u} \right)\left( x \right)} \right) + {D_{{X_{{i_0}}}, - h,1}}\left( {T\left( {{\eta ^4}u} \right)\left( x \right)} \right)} \right]dx} .
\end{align*}
Moreover, it yields from \eqref{eq41} by applying the property of difference quotients \eqref{eq116} that
\begin{align}\label{eq42}
   & \int_\Omega  {\langle {D_{{X_{{i_0}}},h,1}}{\nabla _H}u,{D_{{X_{{i_0}}},h,1}}\left( {{\nabla _H}\left( {{\eta ^4}u} \right)} \right)\rangle dx}  \nonumber \\
   & + \int_\Omega  {\langle {X _{n+i_0}}u,\left[ {{D_{{X_{{i_0}}},h,1}}\left( {T\left( {{\eta ^4}u} \right)\left( x \right)} \right) + {D_{{X_{{i_0}}}, - h,1}}\left( {T\left( {{\eta ^4}u} \right)\left( x \right)} \right)} \right]dx}  \nonumber\\
   & + \frac{{{c_{Q,s}}}}{2}\iint_{\mathbb{H}^n \times \mathbb{H}^n} {\frac{{\left( {{D_{{X_{{i_0}}},h,1}}u\left( x \right) - {D_{{X_{{i_0}}},h,1}}u\left( y \right)} \right)\left( {{D_{{X_{{i_0}}},h,1}}\left( {{\eta ^4}u} \right)\left( x \right) - {D_{{X_{{i_0}}},h,1}}\left( {{\eta ^4}u} \right)\left( y \right)} \right)}}{{\left\| {{y^{ - 1}}\circ x} \right\|_{{\mathbb{H}^n}}^{Q + 2s}}}dxdy} \nonumber \\
  = &  - \int_\Omega  {f{D_{{X_{{i_0}}}, - h,1}}{D_{{X_{{i_0}}},h,1}}\left( {{\eta ^4}u} \right)dx}  .
\end{align}
Referring to the equality below (5.2) in \cite{DA04}, we know
\begin{align}\label{eq43}
  {D_{{X_{{i_0}}},h,1}}\left( {{\nabla _H}\left( {{\eta ^4}u} \right)} \right)\left( x \right) = & {D_{{X_{{i_0}}},h,1}}\left( {4{\eta ^3}{\nabla _H}\eta  \otimes u + {\eta ^4}{\nabla _H}u} \right)\left( x \right) \nonumber\\
  = & 4{D_{{X_{{i_0}}},h,1}}\eta \left( x \right) \cdot \eta^2 {\left( {x{e^{s{X_{{i_0}}}}}} \right)}{\nabla _H}\eta \left( {x{e^{s{X_{{i_0}}}}}} \right) \otimes u\left( {x{e^{s{X_{{i_0}}}}}} \right)\nonumber\\
  & + 4\eta \left( x \right){D_{{X_{{i_0}}},h,1}}\eta \left( x \right)\eta \left( {x{e^{s{X_{{i_0}}}}}} \right){\nabla _H}\eta \left( {x{e^{s{X_{{i_0}}}}}} \right) \otimes u\left( {x{e^{s{X_{{i_0}}}}}} \right)\nonumber\\
  & + 4\eta^2 {\left( x \right)}{D_{{X_{{i_0}}},h,1}}\eta \left( x \right){\nabla _H}\eta \left( {x{e^{s{X_{{i_0}}}}}} \right) \otimes u\left( {x{e^{s{X_{{i_0}}}}}} \right)\nonumber\\
  & + 4\eta^3 {\left( x \right)}{D_{{X_{{i_0}}},h,1}}{\nabla _H}\eta \left( x \right) \otimes u\left( {x{e^{s{X_{{i_0}}}}}} \right)\nonumber\\
  & + 4\eta^3 {\left( x \right)}{\nabla _H}\eta \left( x \right) \otimes {D_{{X_{{i_0}}},h,1}}u\left( x \right)\nonumber\\
  &+ {D_{{X_{{i_0}}},h,1}}\eta \left( x \right)\eta^3 {\left( {x{e^{s{X_{{i_0}}}}}} \right)}{\nabla _H}u\left( {x{e^{s{X_{{i_0}}}}}} \right)\nonumber\\
  & + \eta \left( x \right){D_{{X_{{i_0}}},h,1}}\eta \left( x \right)\eta ^2{\left( {x{e^{s{X_{{i_0}}}}}} \right)}{\nabla _H}u\left( {x{e^{s{X_{{i_0}}}}}} \right)\nonumber\\
  & + \eta^2 {\left( x \right)}{D_{{X_{{i_0}}},h,1}}\eta \left( x \right)\eta \left( {x{e^{s{X_{{i_0}}}}}} \right){\nabla _H}u\left( {x{e^{s{X_{{i_0}}}}}} \right)\nonumber\\
  & + \eta^3 {\left( x \right)}{D_{{X_{{i_0}}},h,1}}\eta \left( x \right){\nabla _H}u\left( {x{e^{s{X_{{i_0}}}}}} \right)\nonumber\\
  & + \eta^4 {\left( x \right)}{D_{{X_{{i_0}}},h,1}}{\nabla _H}u\left( x \right),
\end{align}
and infer by substituting the above relationship into \eqref{eq42} that
\begin{align}\label{eq44}
 & \int_\Omega  {{\eta ^4}{{\left| {{D_{{X_{{i_0}}},h,1}}{\nabla _H}u} \right|}^2}dx} \nonumber \\
& + \frac{{{c_{Q,s}}}}{2}\iint_{\mathbb{H}^n \times \mathbb{H}^n} {\frac{{\left( {{D_{{X_{{i_0}}},h,1}}u\left( x \right) - {D_{{X_{{i_0}}},h,1}}u\left( y \right)} \right)\left( {{D_{{X_{{i_0}}},h,1}}\left( {{\eta ^4}u} \right)\left( x \right) - {D_{{X_{{i_0}}},h,1}}\left( {{\eta ^4}u} \right)\left( y \right)} \right)}}{{\left\| {{y^{ - 1}} \circ x} \right\|_{{\mathbb{H}^n}}^{Q + 2s}}}dxdy}\nonumber \\
=& - 4\int_\Omega  {\langle {D_{{X_{{i_0}}},h,1}}{\nabla _H}u,{D_{{X_{{i_0}}},h,1}}\eta \left( x \right){\eta ^2}\left( {x{e^{s{X_{{i_0}}}}}} \right){\nabla _H}\eta \left( {x{e^{s{X_{{i_0}}}}}} \right) \otimes u\left( {x{e^{s{X_{{i_0}}}}}} \right)\rangle dx}  \nonumber \\
& - 4\int_\Omega  {\langle{D_{{X_{{i_0}}},h,1}}{\nabla _H}u,\eta \left( x \right){D_{{X_{{i_0}}},h,1}}\eta \left( x \right)\eta \left( {x{e^{s{X_{{i_0}}}}}} \right){\nabla _H}\eta \left( {x{e^{s{X_{{i_0}}}}}} \right) \otimes u\left( {x{e^{s{X_{{i_0}}}}}} \right)\rangle dx} \nonumber \\
&- 4\int_\Omega  {\langle{D_{{X_{{i_0}}},h,1}}{\nabla _H}u,\eta ^2{{\left( x \right)}}{D_{{X_{{i_0}}},h,1}}\eta \left( x \right){\nabla _H}\eta \left( {x{e^{s{X_{{i_0}}}}}} \right) \otimes u\left( {x{e^{s{X_{{i_0}}}}}} \right)\rangle dx} \nonumber \\
& - 4\int_\Omega  {\langle{D_{{X_{{i_0}}},h,1}}{\nabla _H}u,\eta^3 {{\left( x \right)}}{D_{{X_{{i_0}}},h,1}}{\nabla _H}\eta \left( x \right) \otimes u\left( {x{e^{s{X_{{i_0}}}}}} \right)\rangle dx} \nonumber \\
& - 4\int_\Omega  {\langle{D_{{X_{{i_0}}},h,1}}{\nabla _H}u,\eta ^3{{\left( x \right)}}{\nabla _H}\eta \left( x \right) \otimes {D_{{X_{{i_0}}},h,1}}u\left( x \right)\rangle dx} \nonumber \\
& - \int_\Omega  {\langle{D_{{X_{{i_0}}},h,1}}{\nabla _H}u,{D_{{X_{{i_0}}},h,1}}\eta \left( x \right)\eta^3 {{\left( {x{e^{s{X_{{i_0}}}}}} \right)}}{\nabla _H}u\left( {x{e^{s{X_{{i_0}}}}}} \right)\rangle dx} \nonumber \\
& - \int_\Omega  {\langle{D_{{X_{{i_0}}},h,1}}{\nabla _H}u,\eta \left( x \right){D_{{X_{{i_0}}},h,1}}\eta \left( x \right)\eta^2 {{\left( {x{e^{s{X_{{i_0}}}}}} \right)}}{\nabla _H}u\left( {x{e^{s{X_{{i_0}}}}}} \right)\rangle dx} \nonumber \\
& - \int_\Omega  {\langle{D_{{X_{{i_0}}},h,1}}{\nabla _H}u,\eta ^2{{\left( x \right)}}{D_{{X_{{i_0}}},h,1}}\eta \left( x \right)\eta \left( {x{e^{s{X_{{i_0}}}}}} \right){\nabla _H}u\left( {x{e^{s{X_{{i_0}}}}}} \right)\rangle dx} \nonumber \\
& - \int_\Omega  {\langle{D_{{X_{{i_0}}},h,1}}{\nabla _H}u,\eta ^3{{\left( x \right)}}{D_{{X_{{i_0}}},h,1}}\eta \left( x \right){\nabla _H}u\left( {x{e^{s{X_{{i_0}}}}}} \right)\rangle dx} \nonumber \\
&+ \int_\Omega  {{D_{{X_{{i_0}}}, - h,1}}{X_{n + {i_0}}}u\cdot T\left( {{\eta ^4}u} \right)dx}  + \int_\Omega  {{D_{{X_{{i_0}}},h,1}}{X_{n + {i_0}}}u\cdot T\left( {{\eta ^4}u} \right)dx} \nonumber \\
   &  - \int_\Omega  {f{D_{{X_{{i_0}}}, - h,1}}{D_{{X_{{i_0}}},h,1}}\left( {{\eta ^4}u} \right)dx}\nonumber \\
: =& \sum\limits_{i = 1}^{11} {{J_i}} .
\end{align}
First, by applying Young's inequality with $\epsilon$, $\left| {{\nabla _H}\eta } \right| \le c,\;\left| {T\eta } \right| \le c$ and Theorem \ref{Th11}, it follows
\begin{align}\label{eq45}
  \sum\limits_{i = 1}^{10} {{J_i}} \le & 10\epsilon \int_\Omega  {{\eta ^4}{{\left| {{D_{{X_{{i_0}}},h,1}}{\nabla _H}u} \right|}^2}dx} + \epsilon \int_\Omega  {{\eta ^4}{{\left| {{D_{{X_{{i_0}}},-h,1}}{\nabla _H}u} \right|}^2}dx}\nonumber \\
   & +c_\epsilon \int_\Omega  {{\eta ^4}{{\left| {Tu} \right|}^2}dx}  + {c_\epsilon }\int_\Omega  {{{\left| {{\nabla _H}\eta } \right|}^4}{{\left| u \right|}^2}dx} \nonumber \\
   &+ {c_\epsilon }\int_\Omega  {{\eta ^2}{{\left| {{\nabla _H}\eta } \right|}^2}\left( {{{\left| u \right|}^2} + {{\left| {{\nabla _H}u} \right|}^2}} \right)dx}  + {c_\epsilon }\int_\Omega  {{\eta ^2}{{\left| {T\eta } \right|}^2}{{\left| u \right|}^2}dx}\nonumber \\
   \le & 10\epsilon \int_\Omega  {{\eta ^4}{{\left| {{D_{{X_{{i_0}}},h,1}}{\nabla _H}u} \right|}^2}dx} + \epsilon \int_\Omega  {{\eta ^4}{{\left| {{D_{{X_{{i_0}}},-h,1}}{\nabla _H}u} \right|}^2}dx}\nonumber \\
   & + c\left\| u \right\|_{H{W^{1,2}}\left( {\mathbb{H}^n} \right)}^2 + c\int_\Omega  {{{\left| f \right|}^2}dx} .
\end{align}
For ${J_{11}}$, we use Young's inequality with $\epsilon$, \eqref{eq43} and $\left| {{\nabla _H}\eta } \right| \le c,\;\left| {T\eta } \right| \le c$ to get
\begin{align}\label{eq46}
  {J_{11}}\le & \epsilon \int_\Omega  {{{\left| {{D_{{X_{{i_0}}}, - h,1}}{D_{{X_{{i_0}}},h,1}}\left( {{\eta ^4}u} \right)} \right|}^2}dx}  + {c_\epsilon }\int_\Omega  {{{\left| f \right|}^2}dx} \nonumber \\
   \le&  \epsilon \int_\Omega  {{{\left| {{D_{{X_{{i_0}}}, - h,1}}{\nabla _H}\left( {{\eta ^4}u} \right)} \right|}^2}dx}  + {c_\epsilon }\int_\Omega  {{{\left| f \right|}^2}dx}\nonumber \\
   \le &\epsilon \int_\Omega  {{\eta ^4}{{\left| {{D_{{X_{{i_0}}}, - h,1}}{\nabla _H}u} \right|}^2}dx}  + c\epsilon \left\| u \right\|_{H{W^{1,2}}\left( {{\mathbb{H}^n}} \right)}^2 + {c_\epsilon }\int_\Omega  {{{\left| f \right|}^2}dx} .
\end{align}
Therefore, by combining \eqref{eq44}-\eqref{eq46}, we have
\begin{align}\label{eq47}
 & \int_\Omega  {{\eta ^4}{{\left| {{D_{{X_{{i_0}}},h,1}}{\nabla _H}u} \right|}^2}dx} \nonumber \\
& + \frac{{{c_{Q,s}}}}{2}\iint_{\mathbb{H}^n \times \mathbb{H}^n} {\frac{{\left( {{D_{{X_{{i_0}}},h,1}}u\left( x \right) - {D_{{X_{{i_0}}},h,1}}u\left( y \right)} \right)\left( {{D_{{X_{{i_0}}},h,1}}\left( {{\eta ^4}u} \right)\left( x \right) - {D_{{X_{{i_0}}},h,1}}\left( {{\eta ^4}u} \right)\left( y \right)} \right)}}{{\left\| {{y^{ - 1}} \circ x} \right\|_{{\mathbb{H}^n}}^{Q + 2s}}}dxdy}\nonumber \\
   \le & 10\epsilon \int_\Omega  {{\eta ^4}{{\left| {{D_{{X_{{i_0}}},h,1}}{\nabla _H}u} \right|}^2}dx} + 2\epsilon \int_\Omega  {{\eta ^4}{{\left| {{D_{{X_{{i_0}}},-h,1}}{\nabla _H}u} \right|}^2}dx}\nonumber \\
    &+ c(\epsilon+1) \left\| u \right\|_{H{W^{1,2}}\left( {{\mathbb{H}^n}} \right)}^2 + {c_\epsilon }\int_\Omega  {{{\left| f \right|}^2}dx} .
\end{align}

Next, we estimate
$${J_{0}}: = \iint_{\mathbb{H}^n \times \mathbb{H}^n} {\frac{{\left( {{D_{{X_{{i_0}}},h,1}}u\left( x \right) - {D_{{X_{{i_0}}},h,1}}u\left( y \right)} \right)\left( {{D_{{X_{{i_0}}},h,1}}\left( {{\eta ^4}u} \right)\left( x \right) - {D_{{X_{{i_0}}},h,1}}\left( {{\eta ^4}u} \right)\left( y \right)} \right)}}{{\left\| {{y^{ - 1}} \circ x} \right\|_{{\mathbb{H}^n}}^{Q + 2s}}}dxdy}.$$
Note that
\begin{align*}
   & {D_{{X_{{i_0}}},h,1}}\left( {{\eta ^4}u} \right)\left( x \right) = \frac{{\left( {{\eta ^4}u} \right)\left( {x{e^{s{X_{{i_0}}}}}} \right) - \left( {{\eta ^4}u} \right)\left( x \right)}}{h} \\
= & \frac{{\left( {{\eta ^4}\left( {x{e^{s{X_{{i_0}}}}}} \right) - {\eta ^4}\left( x \right)} \right)u\left( {x{e^{s{X_{{i_0}}}}}} \right) + {\eta ^4}\left( x \right)\left( {u\left( {x{e^{s{X_{{i_0}}}}}} \right) - u\left( x \right)} \right)}}{h}\\
=& {D_{{X_{{i_0}}},h,1}}\left( {{\eta ^4}} \right)\left( x \right)u\left( {x{e^{s{X_{{i_0}}}}}} \right) + {\eta ^4}\left( x \right){D_{{X_{{i_0}}},h,1}}u\left( x \right),
\end{align*}
so
\begin{align*}
   & {D_{{X_{{i_0}}},h,1}}\left( {{\eta ^4}u} \right)\left( x \right) - {D_{{X_{{i_0}}},h,1}}\left( {{\eta ^4}u} \right)\left( y \right) \\
  = & {D_{{X_{{i_0}}},h,1}}\left( {{\eta ^4}} \right)\left( x \right)u\left( {x{e^{s{X_{{i_0}}}}}} \right) + {\eta ^4}\left( x \right){D_{{X_{{i_0}}},h,1}}u\left( x \right) - \left( {{D_{{X_{{i_0}}},h,1}}\left( {{\eta ^4}} \right)\left( y \right)u\left( {y{e^{s{X_{{i_0}}}}}} \right) + {\eta ^4}\left( y \right){D_{{X_{{i_0}}},h,1}}u\left( y \right)} \right)\\
   =& \left( {{D_{{X_{{i_0}}},h,1}}\left( {{\eta ^4}} \right)\left( x \right)u\left( {x{e^{s{X_{{i_0}}}}}} \right) - {D_{{X_{{i_0}}},h,1}}\left( {{\eta ^4}} \right)\left( y \right)u\left( {y{e^{s{X_{{i_0}}}}}} \right)} \right) + \left( {\left({\eta ^4}{D_{{X_{{i_0}}},h,1}}u\right)\left( x \right) -\left( {\eta ^4}{D_{{X_{{i_0}}},h,1}}u \right)\left( y \right)} \right)\\
    =& {D_{{X_{{i_0}}},h,1}}\left( {{\eta ^4}} \right)\left( x \right)\left( {u\left( {x{e^{s{X_{{i_0}}}}}} \right) - u\left( {y{e^{s{X_{{i_0}}}}}} \right)} \right) + u\left( {y{e^{s{X_{{i_0}}}}}} \right)\left( {{D_{{X_{{i_0}}},h,1}}\left( {{\eta ^4}} \right)\left( x \right) - {D_{{X_{{i_0}}},h,1}}\left( {{\eta ^4}} \right)\left( y \right)} \right)\\
    & + {\eta ^4}\left( x \right)\left( {{D_{{X_{{i_0}}},h,1}}u\left( x \right) - {D_{{X_{{i_0}}},h,1}}u\left( y \right)} \right) + {D_{{X_{{i_0}}},h,1}}u\left( y \right)\left( {{\eta ^4}\left( x \right) - {\eta ^4}\left( y \right)} \right)\\
    =&{\eta ^4}\left( x \right)\left( {{D_{{X_{{i_0}}},h,1}}u\left( x \right) - {D_{{X_{{i_0}}},h,1}}u\left( y \right)} \right) + {D_{{X_{{i_0}}},h,1}}u\left( y \right)\left( {{\eta ^4}\left( x \right) - {\eta ^4}\left( y \right)} \right)\\
    &+u\left( {y{e^{s{X_{{i_0}}}}}} \right)\left( {{D_{{X_{{i_0}}},h,1}}\left( {{\eta ^4}} \right)\left( x \right) - {D_{{X_{{i_0}}},h,1}}\left( {{\eta ^4}} \right)\left( y \right)} \right)+{D_{{X_{{i_0}}},h,1}}\left( {{\eta ^4}} \right)\left( x \right)\left( {u\left( {x{e^{s{X_{{i_0}}}}}} \right) - u\left( {y{e^{s{X_{{i_0}}}}}} \right)} \right) .
\end{align*}
Thus
\begin{align}\label{eq48}
   & {J_{0}} =  \iint_{\mathbb{H}^n \times \mathbb{H}^n} {\frac{{\eta ^4}\left( x \right){\left| {{D_{{X_{{i_0}}},h,1}}u\left( x \right) - {D_{{X_{{i_0}}},h,1}}u\left( y \right)} \right|^2}}{{\left\| {{y^{ - 1}} \circ x} \right\|_{{\mathbb{H}^n}}^{Q + 2s}}}dxdy} \nonumber\\
   &+ \iint_{\mathbb{H}^n \times \mathbb{H}^n} {\frac{{\left( {{D_{{X_{{i_0}}},h,1}}u\left( x \right) - {D_{{X_{{i_0}}},h,1}}u\left( y \right)} \right) {D_{{X_{{i_0}}},h,1}}u\left( y \right)\left( {{\eta ^4}\left( x \right) - {\eta ^4}\left( y \right)}  \right)}}{{\left\| {{y^{ - 1}} \circ x} \right\|_{{\mathbb{H}^n}}^{Q + 2s}}}dxdy} \nonumber\\
  & +\iint_{\mathbb{H}^n \times \mathbb{H}^n} {\frac{{\left( {{D_{{X_{{i_0}}},h,1}}u\left( x \right) - {D_{{X_{{i_0}}},h,1}}u\left( y \right)} \right)u\left( {y{e^{s{X_{{i_0}}}}}} \right)\left( {{D_{{X_{{i_0}}},h,1}}\left( {{\eta ^4}} \right)\left( x \right) - {D_{{X_{{i_0}}},h,1}}\left( {{\eta ^4}} \right)\left( y \right)} \right)}}{{\left\| {{y^{ - 1}} \circ x} \right\|_{{\mathbb{H}^n}}^{Q + 2s}}}dxdy}\nonumber\\
  & +\iint_{\mathbb{H}^n \times \mathbb{H}^n} {\frac{{\left( {{D_{{X_{{i_0}}},h,1}}u\left( x \right) - {D_{{X_{{i_0}}},h,1}}u\left( y \right)} \right){D_{{X_{{i_0}}},h,1}}\left( {{\eta ^4}} \right)\left( x \right)\left( {u\left( {x{e^{s{X_{{i_0}}}}}} \right) - u\left( {y{e^{s{X_{{i_0}}}}}} \right)} \right)}}{{\left\| {{y^{ - 1}} \circ x} \right\|_{{\mathbb{H}^n}}^{Q + 2s}}}dxdy}\nonumber\\
  : =& {J_{01}} + {J_{02}}+ {J_{03}}+ {J_{04}}.
\end{align}

Note that
\begin{align*}
 {J_{01}}& =\iint_{\mathbb{H}^n \times \mathbb{H}^n} {\frac{{{\eta ^4}\left( x \right){{\left| {{D_{{X_{{i_0}}},h,1}}u\left( x \right) - {D_{{X_{{i_0}}},h,1}}u\left( y \right)} \right|}^2}}}{{\left\| {{y^{ - 1}} \circ x} \right\|_{{\mathbb{H}^n}}^{Q + 2s}}}dxdy} \\
  &  =\iint_{\mathbb{H}^n \times \mathbb{H}^n} {\frac{{{\eta ^4}\left( y \right){{\left| {{D_{{X_{{i_0}}},h,1}}u\left( x \right) - {D_{{X_{{i_0}}},h,1}}u\left( y \right)} \right|}^2}}}{{\left\| {{y^{ - 1}} \circ x} \right\|_{{\mathbb{H}^n}}^{Q + 2s}}}dxdy},
\end{align*}
so by applying Young's inequality with $\epsilon '$, it deduces
\begin{align}\label{eq49}
   {J_{02}} =& \iint_{\mathbb{H}^n \times \mathbb{H}^n} {\frac{{\left( {{D_{{X_{{i_0}}},h,1}}u\left( x \right) - {D_{{X_{{i_0}}},h,1}}u\left( y \right)} \right){D_{{X_{{i_0}}},h,1}}u\left( y \right)\left( {{\eta ^4}\left( x \right) - {\eta ^4}\left( y \right)} \right)}}{{\left\| {{y^{ - 1}} \circ x} \right\|_{{\mathbb{H}^n}}^{Q + 2s}}}dxdy} \nonumber \\
   =& \iint_{\mathbb{H}^n \times \mathbb{H}^n} {\frac{{\left( {{D_{{X_{{i_0}}},h,1}}u\left( x \right) - {D_{{X_{{i_0}}},h,1}}u\left( y \right)} \right){D_{{X_{{i_0}}},h,1}}u\left( y \right)\left( {{\eta ^2}\left( x \right) +{\eta ^2}\left( y \right)} \right)\left( {{\eta }\left( x \right) - {\eta }\left( y \right)} \right)}}{{\left\| {{y^{ - 1}} \circ x} \right\|_{{\mathbb{H}^n}}^{Q + 2s}}}} \nonumber \\
   &\times\left( {{\eta }\left( x \right) + {\eta }\left( y \right)} \right)dxdy\nonumber \\
   \le& 2\iint_{\mathbb{H}^n \times \mathbb{H}^n} {\frac{{\left( {{D_{{X_{{i_0}}},h,1}}u\left( x \right) - {D_{{X_{{i_0}}},h,1}}u\left( y \right)} \right){D_{{X_{{i_0}}},h,1}}u\left( y \right)\left( {{\eta ^2}\left( x \right) +{\eta ^2}\left( y \right)} \right)\left( {{\eta }\left( x \right) - {\eta }\left( y \right)} \right)}}{{\left\| {{y^{ - 1}} \circ x} \right\|_{{\mathbb{H}^n}}^{Q + 2s}}}dxdy} \nonumber \\
   =& 2\iint_{\mathbb{H}^n \times \mathbb{H}^n} {\frac{{{\eta ^2}\left( x \right)\left( {{D_{{X_{{i_0}}},h,1}}u\left( x \right) - {D_{{X_{{i_0}}},h,1}}u\left( y \right)} \right){D_{{X_{{i_0}}},h,1}}u\left( y \right)\left( {\eta \left( x \right) - \eta \left( y \right)} \right)}}{{\left\| {{y^{ - 1}} \circ x} \right\|_{\mathbb{H}^n}^{Q + 2s}}}dxdy}\nonumber \\
  & +2\iint_{\mathbb{H}^n \times \mathbb{H}^n} {\frac{{{\eta ^2}\left( y \right)\left( {{D_{{X_{{i_0}}},h,1}}u\left( x \right) - {D_{{X_{{i_0}}},h,1}}u\left( y \right)} \right){D_{{X_{{i_0}}},h,1}}u\left( y \right)\left( {\eta \left( x \right) - \eta \left( y \right)} \right)}}{{\left\| {{y^{ - 1}} \circ x} \right\|_{\mathbb{H}^n}^{Q + 2s}}}dxdy}\nonumber \\
   \le &\epsilon '\iint_{\mathbb{H}^n \times \mathbb{H}^n} {\frac{{{\eta ^4}\left( x \right){{\left| {{D_{{X_{{i_0}}},h,1}}u\left( x \right) - {D_{{X_{{i_0}}},h,1}}u\left( y \right)} \right|}^2}}}{{\left\| {{y^{ - 1}} \circ x} \right\|_{{\mathbb{H}^n}}^{Q + 2s}}}dxdy}\nonumber \\
  & + {c_{\epsilon '}}\iint_{\mathbb{H}^n \times \mathbb{H}^n} {\frac{{{{\left| {{D_{{X_{{i_0}}},h,1}}u\left( y \right)} \right|}^2}{{\left| {\eta \left( x \right) - \eta \left( y \right)} \right|}^2}}}{{\left\| {{y^{ - 1}} \circ x} \right\|_{{\mathbb{H}^n}}^{Q + 2s}}}dxdy}\nonumber \\
  =&\epsilon '{J_{01}} + {c_{\epsilon '}}\iint_{\mathbb{H}^n \times \mathbb{H}^n} {\frac{{{{\left| {{D_{{X_{{i_0}}},h,1}}u\left( y \right)} \right|}^2}{{\left| {\eta \left( x \right) - \eta \left( y \right)} \right|}^2}}}{{\left\| {{y^{ - 1}} \circ x} \right\|_{{\mathbb{H}^n}}^{Q + 2s}}}dxdy}.
\end{align}
Since $\eta  \in C_0^\infty \left( {\mathbb{H}^n}  \right)$, we have that for any $y \in {\mathbb{H}^n}$,
\begin{align*}
   & \int_{{\mathbb{H}^n}} {\frac{{{{\left| {\eta \left( x \right) - \eta \left( y \right)} \right|}^2}}}{{\left\| {{y^{ - 1}} \circ x} \right\|_{{\mathbb{H}^n}}^{Q + 2s}}}dx}  \\
   \le&  \left( {\mathop {\sup }\limits_{{\mathbb{H}^n}} {{\left| {{\nabla _H}\eta } \right|}^2}} \right)\int_{\left\{ {\left\| {{y^{ - 1}} \circ x} \right\|_{{\mathbb{H}^n}}^{} \le 1} \right\}} {\frac{1}{{\left\| {{y^{ - 1}} \circ x} \right\|_{{\mathbb{H}^n}}^{Q + 2\left( {s - 1} \right)}}}dx} \\
   & + 4\left( {\mathop {\sup }\limits_{{\mathbb{H}^n}} {{\left| \eta  \right|}^2}} \right)\int_{\left\{ {\left\| {{y^{ - 1}} \circ x} \right\|_{{\mathbb{H}^n}}^{} > 1} \right\}} {\frac{1}{{\left\| {{y^{ - 1}} \circ x} \right\|_{{\mathbb{H}^n}}^{Q + 2s}}}dx} \\
    = :&c\left( {\eta ,s} \right).
\end{align*}
Moreover, it follows from \eqref{eq49} that
\begin{align}\label{eq410}
   {J_{02}} \le &\epsilon '{J_{01}} + {c_{\epsilon '}}\iint_{\mathbb{H}^n \times \mathbb{H}^n} {\frac{{{{\left| {{D_{{X_{{i_0}}},h,1}}u\left( y \right)} \right|}^2}{{\left| {\eta \left( x \right) - \eta \left( y \right)} \right|}^2}}}{{\left\| {{y^{ - 1}} \circ x} \right\|_{{\mathbb{H}^n}}^{Q + 2s}}}dxdy}\nonumber \\
    \le &\epsilon '{J_{01}}  + c\left( {\epsilon ',\eta ,s} \right)\int_{{\mathbb{H}^n}} {{{\left| {{D_{{X_{{i_0}}},h,1}}u\left( y \right)} \right|}^2}dy}\nonumber \\
     \le &\epsilon '{J_{01}} + c\left( {\epsilon ',\eta ,s} \right)\int_{{\mathbb{H}^n}} {{{\left| {{\nabla _H}u\left( y \right)} \right|}^2}dy} .
\end{align}
Similar to the estimation of $ {J_{02}}$, we have
\begin{equation}\label{eq411}
 {J_{03}} \le \epsilon '{J_{01}} + c\left( {\epsilon ',\eta ,s} \right)\int_{{\mathbb{H}^n}} {{{\left| {u\left( y \right)} \right|}^2}dy}.
\end{equation}

Next, we estimate ${J_{04}}$. Deserve
\begin{align*}
   {D_{{X_{{i_0}}},h,1}}\left( {{\eta ^4}} \right)\left( x \right)  =& {D_{{X_{{i_0}}},h,1}}\eta \left( x \right)\eta^3 {\left( {x{e^{s{X_{{i_0}}}}}} \right)}\nonumber\\
  & + \eta \left( x \right){D_{{X_{{i_0}}},h,1}}\eta \left( x \right)\eta ^2{\left( {x{e^{s{X_{{i_0}}}}}} \right)}\nonumber\\
  & + \eta^2 {\left( x \right)}{D_{{X_{{i_0}}},h,1}}\eta \left( x \right)\eta \left( {x{e^{s{X_{{i_0}}}}}} \right)\nonumber\\
  & + \eta^3 {\left( x \right)}{D_{{X_{{i_0}}},h,1}}\eta \left( x \right),
\end{align*}
it follows by applying Young's inequality with $\epsilon '$ that 
\begin{align}\label{eq412}
   {J_{04}} \le &4\epsilon '{J_{01}} + {c_{\epsilon '}}\iint_{\mathbb{H}^n \times \mathbb{H}^n}{\frac{{{{\left| {{D_{{X_{{i_0}}},h,1}}\eta\left( x \right)} \right|}^2}{{\left| {u \left( x{e^{s{X_{{i_0}}}}} \right) - u \left( y{e^{s{X_{{i_0}}}}} \right)} \right|}^2}}}{{\left\| {{y^{ - 1}} \circ x} \right\|_{{\mathbb{H}^n}}^{Q + 2s}}}dxdy}\nonumber \\
   \le&4\epsilon '{J_{01}} +{c_{\epsilon '}}\iint_{{\mathbb{H}^n} \times \left\{ {{\mathbb{H}^n} \cap \left\{ {{{\left\| {{y^{ - 1}} \circ x} \right\|}_{{\mathbb{H}^n}}} \ge 1} \right\}} \right\}} {\frac{{{{\left| {u\left( x \right) - u\left( y \right)} \right|}^2}}}{{\left\| {{y^{ - 1}} \circ x} \right\|_{{\mathbb{H}^n}}^{Q + 2s}}}dxdy} \nonumber \\
    &+ {c_{\epsilon '}}\iint_{{\mathbb{H}^n} \times \left\{ {{\mathbb{H}^n} \cap \left\{ {{{\left\| {{y^{ - 1}} \circ x} \right\|}_{{\mathbb{H}^n}}} < 1} \right\}} \right\}} {\frac{{{{\left| {u\left( x \right) - u\left( y \right)} \right|}^2}}}{{\left\| {{y^{ - 1}} \circ x} \right\|_{{\mathbb{H}^n}}^{Q + 2s}}}dxdy}   \nonumber\\
   \le &4\epsilon '{J_{01}} + {c(\epsilon ',Q,s)}\left\| u \right\|_{{L^2}\left( {{\mathbb{H}^n}} \right)}^2 + {c_{\epsilon '}}\iint_{{\mathbb{H}^n} \times  {{B_1}} } {\frac{{{{\left| {u\left( x \right) - u\left( {x \circ z} \right)} \right|}^2}}}{{\left\| z \right\|_{{\mathbb{H}^n}}^{Q + 2s}}}d x dz} \nonumber\\
    = &4\epsilon '{J_{01}} + {c(\epsilon ',Q,s)}\left\| u \right\|_{{L^2}\left( {{\mathbb{H}^n}} \right)}^2 + {c_{\epsilon '}}\iint_{{\mathbb{H}^n} \times  {{B_1}} } {{{\left| {\frac{{u\left( x \right) - u\left( {x \circ z} \right)}}{{\left\| z \right\|_{{\mathbb{H}^n}}}}} \right|}^2}\frac{1}{{\left\| z \right\|_{{\mathbb{H}^n}}^{Q + 2\left( {s - 1} \right)}}}dx dz} \nonumber\\
     \le& 4\epsilon '{J_{01}} + {c(\epsilon ',Q,s)}\left\| u \right\|_{{L^2}\left( {{\mathbb{H}^n}} \right)}^2 + {c_{\epsilon '}}\iint_{{\mathbb{H}^n} \times {B_1}} {{{\left| {\int\limits_0^1 {\left| {{\nabla _H}u\left( {x \circ tz} \right)} \right|} dt} \right|}^2}\frac{1}{{\left\| z \right\|_{{\mathbb{H}^n}}^{Q + 2\left( {s - 1} \right)}}}dx dz}\nonumber\\
      \le& 4\epsilon '{J_{01}} + {c(\epsilon ',Q,s)}\left\| u \right\|_{{L^2}\left( {{\mathbb{H}^n}} \right)}^2 + {c_{\epsilon '}}\iint_{{\mathbb{H}^n} \times {B_1}} {\int\limits_0^1 {{{\left| {{\nabla _H}u\left( {x \circ tz} \right)} \right|}^2}} \frac{1}{{\left\| z \right\|_{{\mathbb{H}^n}}^{Q + 2\left( {s - 1} \right)}}}dtdx dz}  \nonumber\\
      =& 4\epsilon '{J_{01}} + {c(\epsilon ',Q,s)}\left\| u \right\|_{{L^2}\left( {{\mathbb{H}^n}} \right)}^2 + {c_{\epsilon '}}\left\| {{\nabla _H}u} \right\|_{{L^2}\left( {{\mathbb{H}^n}} \right)}^2\int_{{B_1}} {\frac{1}{{\left\| z \right\|_{{\mathbb{H}^n}}^{Q + 2\left( {s - 1} \right)}}}dz}    \nonumber\\
      = &4\epsilon '{J_{01}} + {c(\epsilon ',Q,s)}\left\| u \right\|_{H{W^{1,2}}\left( {{\mathbb{H}^n}} \right)}^2.
\end{align}

Therefore, it yields by combining \eqref{eq47}-\eqref{eq412} that
\begin{align*}
  & \int_\Omega  {{\eta ^4}{{\left| {{D_{{X_{{i_0}}},h,1}}{\nabla _H}u} \right|}^2}dx}+\frac{{{c_{Q,s}}}}{2}\left(1-6\epsilon'\right)J_{01}\nonumber\\
  \le & 10\epsilon \int_\Omega  {{\eta ^4}{{\left| {{D_{{X_{{i_0}}},h,1}}{\nabla _H}u} \right|}^2}dx} + 2\epsilon \int_\Omega  {{\eta ^4}{{\left| {{D_{{X_{{i_0}}},-h,1}}{\nabla _H}u} \right|}^2}dx}\nonumber \\
    &+ c\left(\epsilon, \epsilon',Q,s\right)\left\| u \right\|_{H{W^{1,2}}\left( {\mathbb{H}^n} \right)}^2 + {c_\epsilon }\int_\Omega  {{{\left| f \right|}^2}dx}.
\end{align*}
By taking $\epsilon'<\frac{1}{6}$, and using $J_{01}\ge 0$, we have
\begin{align}\label{eq413}
   &  \int_\Omega  {{\eta ^4}{{\left| {{D_{{X_{{i_0}}},h,1}}{\nabla _H}u} \right|}^2}dx} \nonumber\\
 \le & 10\epsilon \int_\Omega  {{\eta ^4}{{\left| {{D_{{X_{{i_0}}},h,1}}{\nabla _H}u} \right|}^2}dx} + 2\epsilon \int_\Omega  {{\eta ^4}{{\left| {{D_{{X_{{i_0}}},-h,1}}{\nabla _H}u} \right|}^2}dx}\nonumber \\
    &+ c\left(\epsilon, Q,s\right)\left\| u \right\|_{H{W^{1,2}}\left( {\mathbb{H}^n} \right)}^2+ {c_\epsilon }\int_\Omega  {{{\left| f \right|}^2}dx}.
\end{align}

Taking the test function in \eqref{eq15}
\[\varphi  = {D_{{X_{{i_0}}},h,1}}{D_{{X_{{i_0}}}, - h,1}}\left( {{\eta ^4}u} \right)\left( x \right),\]
we can get the estimates similarly to \eqref{eq413} for $\int_\Omega  {{\eta ^4}{{\left| {{D_{{X_{{i_0}}},-h,1}}{\nabla _H}u} \right|}^2}dx}$, just needing to replace $x{e^{s{X_{{i_0}}}}}$ with $x{e^{ - s{X_{{i_0}}}}}$. Then we add those estimates, take $\epsilon $ small enough to get
\begin{align}\label{eq414}
    & \int_\Omega  {{\eta ^4}{{\left| {{D_{{X_{{i_0}}},h,1}}{\nabla _H}u} \right|}^2}dx} + \int_\Omega  {{\eta ^4}{{\left| {{D_{{X_{{i_0}}},-h,1}}{\nabla _H}u} \right|}^2}dx}\nonumber \\
    \le& c\left\| u \right\|_{H{W^{1,2}}\left( {\mathbb{H}^n} \right)}^2+ c\int_\Omega  {{{\left| f \right|}^2}dx}.
\end{align}
Passing to the limit as $h \to 0$ via Lemma \ref{Le22}, we have
\[\int_{B\left( {{x_0},\frac{r}{4}} \right)}  {{{\left| {\nabla _H^2u\left( x \right)} \right|}^2}dx} \le\int_\Omega  {{{{\eta ^4}\left| {\nabla _H^2u\left( x \right)} \right|}^2}dx} \le c\left\| u \right\|_{H{W^{1,2}}\left( {\mathbb{H}^n} \right)}^2+ c\int_\Omega  {{{\left| f \right|}^2}dx}.\]

\section{Proof of Theorem \ref{Th13}}

In this section, we prove higher weak differentiability to equation \eqref{eq11}, i.e., Theorem 1.4. However, due to the presence of the non-local term, we cannot directly take derivatives of equation \eqref{eq11}. To overcome this difficulty, we need to combine the truncation discussion with the use of difference quotient, and obtain the following three lemmas and a proposition.

In the sequel, we adopt the following notation: given an arbitrary set $A \subseteq \mathbb{H}^n$ and a constant $d > 0$, we define the $d$-neighborhood of $A$ as
\begin{equation*}
A_d := \{x \in \mathbb{H}^n : \text{dist}(x, A) < d\}.
\end{equation*}
Moreover, for any $w \in L^2(\mathbb{H}^n)$, we define the finite difference operator in the $X_i$-direction as
\begin{equation*}
w_h(x) := D_{X_i,h} w(x) = \frac{w(xe^{hX_i}) - w(x)}{h},
\end{equation*}
where $h \in \mathbb{R} \setminus \{0\}$ and $i \in \{1, \dots, 2n\}$.

\begin{lemma}\label{Le51}
Let $u \in H{W^{1,2}}\left( {{\mathbb{H}^n}} \right)$ be a weak solution to \eqref{eq11}, and $\Omega '$ be an open subset with $\bar \Omega ' \subseteq \Omega\subseteq{\mathbb{H}^n}$. Write $\rho  = {\rm{dist}}\left( {\Omega ',\partial \Omega } \right) > 0$. If $\left| h \right| < \rho $, then ${u_h}$ is a weak solution to the equation
\begin{equation}\label{eq52}
 - \Delta_{\mathbb{H}^n} {u_h}\left( x \right) + {c_{Q,s}}P.V.\int_{{\mathbb{H}^n}} {\frac{{{u_h}\left( x \right) - {u_h}\left( y \right)}}{{\left\| {{y^{ - 1}} \circ x} \right\|_{{\mathbb{H}^n}}^{Q + 2s}}}dy}  = {f_h}\left( x \right),\;x \in \Omega '.
\end{equation}
\end{lemma}

\textbf{Proof.} It is noted that the formula
\begin{equation}\label{eq53}
\int_{\Omega '} {{\nabla _H}{u_h}\cdot{\nabla _H}\varphi dx}  = \int_{{\mathbb{H}^n}} {{\nabla _H}{u_h}\cdot{\nabla _H}\varphi dx}  =  - \int_{{\mathbb{H}^n}} {{\nabla _H}u\cdot{\nabla _H}{\psi _{ - h}}dx} 
\end{equation}
holds for any $\varphi  \in C_0^\infty \left( {\Omega ',\mathbb{R}} \right)$, where ${\psi _{ - h}}\left( x \right) =  - \frac{{\varphi \left( {x{e^{ - hZ}}} \right) - \varphi \left( x \right)}}{h}$. Due to ${\rm{supp}}\left( \varphi  \right) \subseteq \Omega '$, it follows
\[{\rm{supp}}\left( {\varphi \left( {\bullet{e^{ - hZ}}} \right)} \right) \subseteq \Omega ' \circ {e^{hZ}} \subseteq \left\{ {x \in {\mathbb{H}^n}:{\rm{dist}}\left( {x,\Omega '} \right) < \left| h \right|} \right\} = :{\Omega '_h}.\]
Hence, if $\left| h \right| < \rho $, then

(1) ${\psi _{ - h}} \in C_0^\infty \left( {{\mathbb{H}^n},\mathbb{R}} \right);$

(2) ${\rm{supp}}\left( {{\psi _{ - h}}} \right) \subseteq {\Omega '_h} \subseteq \Omega .$\\
In particular,
\begin{equation}\label{eq54}
{\psi _{ - h}} \in C_0^\infty \left( {{{\Omega '}_h},\mathbb{R}} \right).
\end{equation}
Therefore, \eqref{eq53} can be written as
\begin{equation}\label{eq55}
\int_{\Omega '} {{\nabla _H}{u_h}\cdot{\nabla _H}\varphi dx}  =  - \int_\Omega  {{\nabla _H}u\cdot{\nabla _H}{\psi _{ - h}}dx} .
\end{equation}

We note that for any $h \in \mathbb{R}$,
\begin{equation}\label{eq56}
\iint_{\mathbb{H}^n \times \mathbb{H}^n} {\frac{{\left( {{u_h}\left( x \right) - {u_h}\left( y \right)} \right)\left( {\varphi \left( x \right) - \varphi \left( y \right)} \right)}}{{\left\| {{y^{ - 1}} \circ x} \right\|_{{\mathbb{H}^n}}^{Q + 2s}}}dxdy}  =  - \iint_{\mathbb{H}^n \times \mathbb{H}^n} {\frac{{\left( {u\left( x \right) - u\left( y \right)} \right)\left( {{\psi _{ - h}}\left( x \right) - {\psi _{ - h}}\left( y \right)} \right)}}{{\left\| {{y^{ - 1}} \circ x} \right\|_{{\mathbb{H}^n}}^{Q + 2s}}}dxdy} ,
\end{equation}
so by combining \eqref{eq55} with \eqref{eq56}, and noting that $u$ is a weak solution of \eqref{eq11}, it follows
\begin{align}\label{eq57}
   & \int_{\Omega '} {{\nabla _H}{u_h}\cdot{\nabla _H}\varphi dx}  + \frac{{{c_{Q,s}}}}{2}\iint_{\mathbb{H}^n \times \mathbb{H}^n} {\frac{{\left( {{u_h}\left( x \right) - {u_h}\left( y \right)} \right)\left( {\varphi \left( x \right) - \varphi \left( y \right)} \right)}}{{\left\| {{y^{ - 1}} \circ x} \right\|_{{\mathbb{H}^n}}^{Q + 2s}}}dxdy} \nonumber\\
  = &   - \int_\Omega  {{\nabla _H}u\cdot{\nabla _H}{\psi _{ - h}}dx}  - \frac{{{c_{Q,s}}}}{2}\iint_{\mathbb{H}^n \times \mathbb{H}^n} {\frac{{\left( {u\left( x \right) - u\left( y \right)} \right)\left( {{\psi _{ - h}}\left( x \right) - {\psi _{ - h}}\left( y \right)} \right)}}{{\left\| {{y^{ - 1}} \circ x} \right\|_{{\mathbb{H}^n}}^{Q + 2s}}}dxdy}\nonumber\\
 = & - \int_{{{\Omega '}_\rho }} {f{\psi _{ - h}}dx} .
\end{align}

If $\left| h \right| < \rho $, then
\[{\rm{supp}}\left( \varphi  \right) \subseteq \Omega ' \subseteq {\Omega '_h} \cap \left( {{{\Omega '}_h} \circ {e^{ - hZ}}} \right){\rm{,}}\]
and so
\begin{align}\label{eq58}
   &  - \int_{{{\Omega '}_\rho }} {f{\psi _{ - h}}dx}  = \int_{{{\Omega '}_\rho }} {f\left( x \right)\left( {\frac{{\varphi \left( {x{e^{ - hZ}}} \right) - \varphi \left( x \right)}}{h}} \right)dx}\nonumber\\
  = & \frac{1}{h}\left( {\int_{{{\Omega '}_\rho } \circ {e^{ - hZ}}} {f\left( {x{e^{hZ}}} \right)\varphi \left( x \right)dx}  - \int_{{{\Omega '}_\rho }} {f\left( x \right)\varphi \left( x \right)dx} } \right) = \int_{\Omega '} {{f_h}\varphi dx}.
\end{align}
By combining \eqref{eq57} with \eqref{eq58}, we get
\[\int_{\Omega '} {{\nabla _H}{u_h}\cdot{\nabla _H}\varphi dx}  + \frac{{{c_{Q,s}}}}{2}\iint_{\mathbb{H}^n \times \mathbb{H}^n} {\frac{{\left( {{u_h}\left( x \right) - {u_h}\left( y \right)} \right)\left( {\varphi \left( x \right) - \varphi \left( y \right)} \right)}}{{\left\| {{y^{ - 1}} \circ x} \right\|_{{\mathbb{H}^n}}^{Q + 2s}}}dxdy}  = \int_{\Omega '} {{f_h}\varphi dx} \]
for any $\varphi  \in C_0^\infty \left( {\Omega ',\mathbb{R}} \right)$. Therefore, $u_h$ is a weak solution of \eqref{eq52}.

\begin{proposition}\label{Pro52}
Let $m \in \mathbb{N} \cup \left\{ 0 \right\},\;f \in H{W^{m,2}}\left( \Omega  \right)$. If $u \in H{W^{m + 1,2}}\left( {{\mathbb{H}^n}} \right)$ is a weak solution to \eqref{eq11}, then
$u \in HW_{\mathrm{loc}}^{m + 2,2}\left( \Omega  \right)$. Moreover, for any open set $V$ with $V \subseteq \bar V \subseteq \Omega \subseteq{\mathbb{H}^n}$, there exists a constant ${C_m} > 0$ independent $h$ such that
\begin{equation}\label{eq59}
{\left\| u \right\|_{H{W^{m + 2,2}}\left( V \right)}} \le {C_m}\left( {{{\left\| f \right\|}_{H{W^{m,2}}\left( \Omega  \right)}} + {{\left\| u \right\|}_{H{W^{m + 1,2}}\left( {{{\mathbb{H}}^n}} \right)}}} \right).
\end{equation}
\end{proposition}

\textbf{Proof.} We proceed by induction on $m \in \mathbb{N} \cup \left\{ 0 \right\}$. Firstly, for the case of $m=0$(that is $f \in L^2(\Omega)$ and $u \in HW^{1,2}({\mathbb{H}^n})$) the result is given by Theorem \ref{Th12}. Then we assume that Proposition \ref{Pro52} holds for some $m>0$, and we prove that it still holds for $m+1$.

Let $V$ be a fixed open set with $V \subseteq \bar V \subseteq \Omega $, $\Omega '$ be an open subset of $\Omega$ with $\bar V \subseteq \Omega ' \subseteq \bar \Omega ' \subseteq \Omega$, and $f \in H{W^{m + 1,2}}\left( \Omega  \right)$, $u \in H{W^{m + 2,2}}\left( {{\mathbb{H}^n}} \right)$ be a weak solution of \eqref{eq11}. We set $\rho  = {\rm{dist}}\left( {\Omega ',\partial \Omega } \right) > 0$. If $\left| h \right| < \rho $, then $u_h$ is a weak solution of \eqref{eq52} from Lemma \ref{Le51}. Because of $u \in H{W^{m + 2,2}}\left( {{\mathbb{H}^n}} \right)$, we have $u_h \in H{W^{m + 1,2}}\left( {{\mathbb{H}^n}} \right)$. Then by using the inductive hypothesis to $u_h$, we obtain

(1) ${u_h} \in HW_{\mathrm{loc}}^{m + 2,2}\left( {\Omega '} \right)$. Particularly, ${u_h} \in H{W^{m + 2,2}}\left( V \right)$;

(2) there exists a constant ${C_m} > 0$ independent of $h$ such that
\[{\left\| {{u_h}} \right\|_{H{W^{m + 2,2}}\left( V \right)}} \le {C_m}\left( {{{\left\| {{f_h}} \right\|}_{H{W^{m,2}}\left( {\Omega '} \right)}} + {{\left\| {{u_h}} \right\|}_{H{W^{m + 1,2}}\left( {{{\rm{H}}^n}} \right)}}} \right)\]
for $\left| h \right| < \rho $.

Note that $f \in H{W^{m + 1,2}}\left( \Omega  \right)$ and $u \in H{W^{m + 2,2}}\left( {{\mathbb{H}^n}} \right)$, so it holds
\[{\left\| {{f_h}} \right\|_{H{W^{m,2}}\left( {\Omega '} \right)}} \le c{\left\| f \right\|_{H{W^{m + 1,2}}\left( {\Omega '} \right)}},\;\hbox{and}\;{\left\| {{u_h}} \right\|_{H{W^{m + 1,2}}\left( {{{\rm{H}}^n}} \right)}} \le c{\left\| u \right\|_{H{W^{m + 2,2}}\left( {{{\rm{H}}^n}} \right)}},\]
where $c$ is independent of $h$. Then we derive
\[{\left\| {{u_h}} \right\|_{H{W^{m + 2,2}}\left( V \right)}} \le c{C_m}\left( {{{\left\| f \right\|}_{H{W^{m + 1,2}}\left( {\Omega '} \right)}} + {{\left\| u \right\|}_{H{W^{m + 2,2}}\left( {{{\mathbb{H}}^n}} \right)}}} \right),\]
and this estimate is uniform with respect to $\left| h \right| < \rho $. Hence, $u \in H{W^{m + 3,2}}\left( V \right)$ and
\[{\left\| u \right\|_{H{W^{m + 3,2}}\left( V \right)}} \le c{C_m}\left( {{{\left\| f \right\|}_{H{W^{m + 1,2}}\left( {\Omega '} \right)}} + {{\left\| u \right\|}_{H{W^{m + 2,2}}\left( {{{\rm{H}}^n}} \right)}}} \right).\]
The proof is completed.

To remove the condition $u \in H{W^{m + 1,2}}\left( {{\mathbb{H}^n}} \right)$ in Proposition \ref{Pro52}, we need to conduct a truncation discussion. To do so, we first give the following two lemmas.

\begin{lemma}\label{Le53}
Let the open set $\Omega ' \subseteq {\mathbb{H}^n},\;\delta  > 0,\;\alpha  > Q,\;g \in {L^2}\left( {{\mathbb{H}^n}} \right)$ satisfy
\begin{equation}\label{eq510}
g\left( x \right) = 0\;\hbox{for}\;\hbox{a.e.}\;x \in {\Omega '_\delta },
\end{equation}
where ${\Omega '_\delta } = \left\{ {x \in {\mathbb{H}^n}:{\rm{dist}}\left( {x,\Omega '} \right) < \delta } \right\}$. Then, the following conclusions hold:

\rm{(1)} for any fixed $x \in {\Omega '_{{\raise0.5ex\hbox{$\scriptstyle \delta $}
\kern-0.1em/\kern-0.15em
\lower0.25ex\hbox{$\scriptstyle 2$}}}}$, we have
\begin{equation}\label{eq511}
 y \mapsto \frac{{g\left( y \right)}}{{\left\| {{y^{ - 1}} \circ x} \right\|_{{\mathbb{H}^n}}^\alpha }} \in {L^1}\left( {{\mathbb{H}^n}} \right);
\end{equation}

\rm{(2)} if we define the functional $ {G_\alpha }\left[ g \right]\left( x \right)$ as
\begin{equation}\label{eq512}
 {G_\alpha }\left[ g \right]\left( x \right) = \int_{{{\mathbb{H}}^n}} {\frac{{g\left( y \right)}}{{\left\| {{y^{ - 1}} \circ x} \right\|_{{\mathbb{H}^n}}^\alpha }}dy} ,
\end{equation}
then ${G_\alpha }\left[ g \right] \in {C^\infty }\left( {{{\Omega '}_{{\raise0.5ex\hbox{$\scriptstyle \delta $}
\kern-0.1em/\kern-0.15em
\lower0.25ex\hbox{$\scriptstyle 2$}}}}} \right).$
\end{lemma}

\textbf{Proof.} We first prove the conclusion (1). Let $x \in {\Omega '_{{\raise0.5ex\hbox{$\scriptstyle \delta $}
\kern-0.1em/\kern-0.15em
\lower0.25ex\hbox{$\scriptstyle 2$}}}}$, then it yields from \eqref{eq510} and H\"{o}lder's inequality that
\begin{align}\label{eq513}
   & \int_{{\mathbb{H}^n}} {\frac{{\left| {g\left( y \right)} \right|}}{{\left\| {{y^{ - 1}} \circ x} \right\|_{{\mathbb{H}^n}}^\alpha }}dy}  = \int_{{\mathbb{H}^n}\backslash {{\Omega '}_\delta }} {\frac{{\left| {g\left( y \right)} \right|}}{{\left\| {{y^{ - 1}} \circ x} \right\|_{{\mathbb{H}^n}}^\alpha }}dy}\nonumber  \\
 \le  &  {\left( {\int_{{\mathbb{H}^n}\backslash {{\Omega '}_\delta }} {\frac{{{{\left| {g\left( y \right)} \right|}^2}}}{{\left\| {{y^{ - 1}} \circ x} \right\|_{{\mathbb{H}^n}}^\alpha }}dy} } \right)^{{\raise0.5ex\hbox{$\scriptstyle 1$}
\kern-0.1em/\kern-0.15em
\lower0.25ex\hbox{$\scriptstyle 2$}}}}{\left( {\int_{{\mathbb{H}^n}\backslash {{\Omega '}_\delta }} {\frac{1}{{\left\| {{y^{ - 1}} \circ x} \right\|_{{\mathbb{H}^n}}^\alpha }}dy} } \right)^{{\raise0.5ex\hbox{$\scriptstyle 1$}
\kern-0.1em/\kern-0.15em
\lower0.25ex\hbox{$\scriptstyle 2$}}}}.
\end{align}
Again because $x \in {\Omega '_{{\raise0.5ex\hbox{$\scriptstyle \delta $}
\kern-0.1em/\kern-0.15em
\lower0.25ex\hbox{$\scriptstyle 2$}}}}$, one has
\begin{equation}\label{eq514}
{\left\| {{y^{ - 1}} \circ x} \right\|_{{\mathbb{H}^n}}} \ge \frac{\delta }{2}
\end{equation}
for any $y \in {\mathbb{H}^n}\backslash {\Omega '_\delta }$. Then, by combining \eqref{eq513} with \eqref{eq514} and applying $\alpha> Q,\;g \in {L^2}\left( {{{\mathbb{H}}^n}} \right)$,
it deduces
\begin{align*}
  \int_{{\mathbb{H}^n}} {\frac{{\left| {g\left( y \right)} \right|}}{{\left\| {{y^{ - 1}} \circ x} \right\|_{{\mathbb{H}^n}}^\alpha }}dy}  \le & {\left( {\frac{2}{\delta }} \right)^{{\raise0.5ex\hbox{$\scriptstyle \alpha $}
\kern-0.1em/\kern-0.15em
\lower0.25ex\hbox{$\scriptstyle 2$}}}}{\left( {\int_{{\mathbb{H}^n}\backslash {{\Omega '}_\delta }} {{{\left| {g\left( y \right)} \right|}^2}dy} } \right)^{{\raise0.5ex\hbox{$\scriptstyle 1$}
\kern-0.1em/\kern-0.15em
\lower0.25ex\hbox{$\scriptstyle 2$}}}}{\left( {\int_{\left\{ {{{\left\| {{y^{ - 1}} \circ x} \right\|}_{{\mathbb{H}^n}}} \ge \frac{\delta }{2}} \right\}} {\frac{1}{{\left\| {{y^{ - 1}} \circ x} \right\|_{{\mathbb{H}^n}}^\alpha }}dy} } \right)^{{\raise0.5ex\hbox{$\scriptstyle 1$}
\kern-0.1em/\kern-0.15em
\lower0.25ex\hbox{$\scriptstyle 2$}}}}\\
\le  &  {\left( {\frac{2}{\delta }} \right)^{{\raise0.5ex\hbox{$\scriptstyle \alpha $}
\kern-0.1em/\kern-0.15em
\lower0.25ex\hbox{$\scriptstyle 2$}}}}{\left( {\int_{\left\{ {{{\left\| z \right\|}_{{\mathbb{H}^n}}} \ge \frac{\delta }{2}} \right\}} {\frac{1}{{\left\| z \right\|_{{\mathbb{H}^n}}^\alpha }}dy} } \right)^{{\raise0.5ex\hbox{$\scriptstyle 1$}
\kern-0.1em/\kern-0.15em
\lower0.25ex\hbox{$\scriptstyle 2$}}}}{\left\| g \right\|_{{L^2}\left( {{\mathbb{H}^n}} \right)}} < \infty .
\end{align*}
Hence, \eqref{eq511} is proved.

Next, we prove the conclusion (2), i.e., we need to prove that for every fixed $x \in {\Omega '_{{\raise0.5ex\hbox{$\scriptstyle \delta $}
\kern-0.1em/\kern-0.15em
\lower0.25ex\hbox{$\scriptstyle 2$}}}}$ and every $2n+1$-tuple $\gamma  = \left( {{\gamma _1},{\gamma _2}, \cdots ,{\gamma _{2n + 1}}} \right) \in {\left( { \mathbb{N}\cup \left\{ 0 \right\}} \right)^{2n + 1}}$, it holds
\begin{equation}\label{eq515}
\partial _x^\gamma {G_\alpha }\left[ g \right]\left( x \right) = \int_{{\mathbb{H} ^n}\backslash {{\Omega '}_\delta }} {\partial _x^\gamma \left( {\frac{{g\left( y \right)}}{{\left\| {{y^{ - 1}} \circ x} \right\|_{{\mathbb{H} ^n}}^\alpha }}} \right)dy} .
\end{equation}

Note that for any $x \ne y \in {\mathbb{H}^n}$, there holds
\begin{equation}\label{eq516}
\partial _x^\gamma \left( {\frac{{g\left( y \right)}}{{\left\| {{y^{ - 1}} \circ x} \right\|_{{\mathbb{H}^n}}^\alpha }}} \right) \le c\left( \alpha  \right)\frac{{\left| {g\left( y \right)} \right|}}{{\left\| {{y^{ - 1}} \circ x} \right\|_{{\mathbb{H}^n}}^{\alpha  + \left| \gamma  \right|}}}.
\end{equation}
Then, from the conclusion (1), it infers that for every fixed $x \in {\Omega '_{{\raise0.5ex\hbox{$\scriptstyle \delta $}
\kern-0.1em/\kern-0.15em
\lower0.25ex\hbox{$\scriptstyle 2$}}}}$, 
\begin{equation}\label{eq517}
y \mapsto \partial _x^\gamma \left( {\frac{{g\left( y \right)}}{{\left\| {{y^{ - 1}} \circ x} \right\|_{{\mathbb{H}^n}}^\alpha }}} \right) \in {L^1}\left( {{\mathbb{H}^n}} \right).
\end{equation}
Next, we prove the claim that: for every ${x_0} \in {\Omega '_{{\raise0.5ex\hbox{$\scriptstyle \delta $}
\kern-0.1em/\kern-0.15em
\lower0.25ex\hbox{$\scriptstyle 2$}}}},\;\gamma  = \left( {{\gamma _1},{\gamma _2}, \cdots ,{\gamma _{2n + 1}}} \right) \in {\left( { \mathbb{N}\cup \left\{ 0 \right\}} \right)^{2n + 1}}$, there exist $r > 0$ and ${\Theta _{\gamma ,{x_0},r}}\left( y \right) \in {L^1}\left( {{\mathbb{H}^n}} \right)$ such that
$$\partial _x^\gamma \left( {\frac{{g\left( y \right)}}{{\left\| {{y^{ - 1}} \circ x} \right\|_{{\mathbb{H}^n}}^\alpha }}} \right) \le  {\Theta _{\gamma ,{x_0},r}}\left( y \right).$$

In fact, assume ${x_0} \in {\Omega '_{{\raise0.5ex\hbox{$\scriptstyle \delta $}
\kern-0.1em/\kern-0.15em
\lower0.25ex\hbox{$\scriptstyle 2$}}}},\;\gamma  = \left( {{\gamma _1},{\gamma _2}, \cdots ,{\gamma _{2n + 1}}} \right) \in {\left( { \mathbb{N}\cup \left\{ 0 \right\}} \right)^{2n + 1}}$ and take $r > 0$ such that $B\left( {{x_0},r} \right) \subseteq {\Omega '_{{\raise0.5ex\hbox{$\scriptstyle \delta $}
\kern-0.1em/\kern-0.15em
\lower0.25ex\hbox{$\scriptstyle 2$}}}}$. On the one hand, for any $x \in B\left( {{x_0},r} \right),\;y \in \left( {{\mathbb{H}^n}\backslash {{\Omega '}_\delta }} \right) \cap B\left( {{x_0},2r} \right)$, we use \eqref{eq514} to obtain
\begin{equation}\label{eq518}
\frac{{{{\left\| {{y^{ - 1}} \circ x} \right\|}_{{\mathbb{H}^n}}}}}{{{{\left\| {{y^{ - 1}} \circ {x_0}} \right\|}_{{\mathbb{H}^n}}}}} \ge \frac{{\frac{\delta }{2}}}{{2r}} = \frac{\delta }{{4r}}.
\end{equation}
On the other hand, for any $x \in B\left( {{x_0},r} \right),\;y \in \left( {{\mathbb{H}^n}\backslash {{\Omega '}_\delta }} \right)\backslash B\left( {{x_0},2r} \right)$, it follows
\[{\left\| {x_0^{ - 1} \circ x} \right\|_{{\mathbb{H}^n}}} < r < \frac{1}{2}{\left\| {{y^{ - 1}} \circ {x_0}} \right\|_{{\mathbb{H}^n}}},\]
so it yields by using the triangle inequality that
\begin{equation}\label{eq519}
\frac{{{{\left\| {{y^{ - 1}} \circ x} \right\|}_{{\mathbb{H} ^n}}}}}{{{{\left\| {{y^{ - 1}} \circ {x_0}} \right\|}_{{\mathbb{H} ^n}}}}} \ge \frac{{{{\left\| {{y^{ - 1}} \circ {x_0}} \right\|}_{{\mathbb{H} ^n}}} - {{\left\| {x_0^{ - 1} \circ x} \right\|}_{{\mathbb{H} ^n}}}}}{{{{\left\| {{y^{ - 1}} \circ {x_0}} \right\|}_{{\mathbb{H} ^n}}}}} = 1 - \frac{{{{\left\| {x_0^{ - 1} \circ x} \right\|}_{{\mathbb{H} ^n}}}}}{{{{\left\| {{y^{ - 1}} \circ {x_0}} \right\|}_{{\mathbb{H} ^n}}}}} \ge \frac{1}{2}.
\end{equation}
Then, by combining \eqref{eq518} with \eqref{eq519}, one has
\begin{equation}\label{eq520}
 \frac{{{{\left\| {{y^{ - 1}} \circ x} \right\|}_{{\mathbb{H}^n}}}}}{{{{\left\| {{y^{ - 1}} \circ {x_0}} \right\|}_{{\mathbb{H}^n}}}}} \ge \min \left\{ {\frac{\delta }{{4r}},\frac{1}{2}} \right\}: = \kappa
\end{equation}
for any $x \in B\left( {{x_0},r} \right),\;y \in {\mathbb{H}^n}\backslash {\Omega '_\delta }$.

Now, by combining \eqref{eq516} with \eqref{eq520}, it follows
\begin{equation}\label{eq521}
\partial _x^\gamma \left( {\frac{{g\left( y \right)}}{{\left\| {{y^{ - 1}} \circ x} \right\|_{{\mathbb{H}^n}}^\alpha }}} \right) \le \frac{{c\left( \alpha  \right)}}{{{\kappa ^{\alpha  + \left| \gamma  \right|}}}}\frac{{\left| {g\left( y \right)} \right|}}{{\left\| {{y^{ - 1}} \circ {x_0}} \right\|_{{\mathbb{H}^n}}^{\alpha  + \left| \gamma  \right|}}}: = {\Theta _{\gamma ,{x_0},r}}\left( y \right)
\end{equation}
for any $x \in B\left( {{x_0},r} \right),\;y \in {\mathbb{H}^n}\backslash {\Omega '_\delta }$. Also, because of ${x_0} \in {\Omega '_{{\raise0.5ex\hbox{$\scriptstyle \delta $}
\kern-0.1em/\kern-0.15em
\lower0.25ex\hbox{$\scriptstyle 2$}}}}$, we have ${\Theta _{\gamma ,{x_0},r}}\left( y \right) \in {L^1}\left( {{\mathbb{H}^n}} \right)$ from the conclusion (1). Therefore, \eqref{eq515} is established, and the conclusion (2) is confirmed.

\begin{lemma}\label{Le54}
Let $u \in H{W^{1,2}}\left( {{\mathbb{H}^n}} \right)$ be a weak solution to \eqref{eq11} and the open set $\Omega '$ with $\bar \Omega ' \subseteq \Omega \subseteq{\mathbb{H}^n}$. If for $\rho  = {\rm{dist}}\left( {\Omega ',\partial \Omega } \right),\;\xi  \in C_0^\infty \left( {{\mathbb{H}^n},\mathbb{R}} \right)$ satisfies

\rm{(1)} $\xi  \equiv 1$ on ${\Omega '_{{\raise0.5ex\hbox{$\scriptstyle \rho $}
\kern-0.1em/\kern-0.15em
\lower0.25ex\hbox{$\scriptstyle 4$}}}}$;

\rm{(2)} ${\rm{supp}}\xi  \subseteq {\Omega '_{{\raise0.5ex\hbox{$\scriptstyle \rho $}
\kern-0.1em/\kern-0.15em
\lower0.25ex\hbox{$\scriptstyle 2$}}}}$;

\rm{(3)} $0 \le \xi  \le 1$ on $\mathbb{H}^n$,\\
then there exists $\psi  \in {C^\infty }\left( {\bar \Omega '}, \mathbb{R}\right)$ such that $\upsilon : = u\xi $ is a weak solution to the equation
\begin{equation}\label{eq522}
 - \Delta_{{\mathbb{H}^n}} \upsilon \left( x \right) + {c_{Q,s}}P.V.\int_{{\mathbb{H}^n}} {\frac{{\upsilon \left( x \right) - \upsilon \left( y \right)}}{{\left\| {{y^{ - 1}} \circ x} \right\|_{{\mathbb{H}^n}}^{Q + 2s}}}dy}  = f\left( x \right) + \psi \left( x \right),\;x \in \Omega '.
\end{equation}
\end{lemma}

\textbf{Proof.} Note $\upsilon  \in H{W^{1,2}}\left( {{\mathbb{H}^n}} \right)$ from $u \in H{W^{1,2}}\left( {{\mathbb{H}^n}} \right)\;\hbox{and}\;\xi  \in C_0^\infty \left( {{\mathbb{H}^n},\mathbb{R}} \right)$. If we write $\omega  = u\left( {1 - \xi } \right)$, then $\omega  = u - u\xi  = u - \upsilon $ and $\omega ,\left| \omega  \right| \in H{W^{1,2}}\left( {{\mathbb{H}^n}} \right)$. Also, because of $\xi  \equiv 1$ on ${\Omega '_{{\raise0.5ex\hbox{$\scriptstyle \rho $}
\kern-0.1em/\kern-0.15em
\lower0.25ex\hbox{$\scriptstyle 4$}}}}$, it hold
\begin{equation}\label{eq523}
 \omega  = u\left( {1 - \xi } \right) = 0\;\hbox{on}\;{\Omega '_{{\raise0.5ex\hbox{$\scriptstyle \rho $}
\kern-0.1em/\kern-0.15em
\lower0.25ex\hbox{$\scriptstyle 4$}}}}
\end{equation}
and
\begin{equation}\label{eq524}
 u \equiv \upsilon \;\hbox{on}\;\Omega '.
\end{equation}
Then, according to the definition of weak solutions and \eqref{eq524}, it deduces for any $\varphi  \in C_0^\infty \left( {\Omega ',\mathbb{R}} \right)$ 
\begin{align}\label{eq525}
   \int_{\Omega '} {f\varphi dx}  =& \int_{\Omega '} {{\nabla _H}u \cdot {\nabla _H}\phi dx}  + \frac{{{c_{Q,s}}}}{2}\iint_{\mathbb{H}^n \times \mathbb{H}^n} {\frac{{\left( {u\left( x \right) - u\left( y \right)} \right)\left( {\varphi \left( x \right) - \varphi \left( y \right)} \right)}}{{\left\| {{y^{ - 1}} \circ x} \right\|_{{\mathbb{H}^n}}^{Q + 2s}}}dxdy}  \nonumber\\
  = & \int_{\Omega '} {{\nabla _H}\upsilon  \cdot {\nabla _H}\phi dx}  + \frac{{{c_{Q,s}}}}{2}\iint_{\mathbb{H}^n \times \mathbb{H}^n} {\frac{{\left( {\upsilon \left( x \right) - \upsilon \left( y \right)} \right)\left( {\varphi \left( x \right) - \varphi \left( y \right)} \right)}}{{\left\| {{y^{ - 1}} \circ x} \right\|_{{\mathbb{H}^n}}^{Q + 2s}}}dxdy}  \nonumber\\
  & + \frac{{{c_{Q,s}}}}{2}\iint_{\mathbb{H}^n \times \mathbb{H}^n} {\frac{{\left( {\omega \left( x \right) - \omega \left( y \right)} \right)\left( {\varphi \left( x \right) - \varphi \left( y \right)} \right)}}{{\left\| {{y^{ - 1}} \circ x} \right\|_{{\mathbb{H}^n}}^{Q + 2s}}}dxdy} .
\end{align}

Moreover, the conditions of Lemma \ref{Le53} are met with $\alpha : = Q + 2s > Q,\;\delta  = \frac{\rho }{4},\;g: = \omega $ or $g: = \left| \omega  \right|$ from \eqref{eq523}, and so
\begin{equation}\label{eq526}
  {G_{Q + 2s}}\left[ \omega  \right]\left( x \right) = \int_{{\mathbb{H}^n}} {\frac{{\omega \left( y \right)}}{{\left\| {{y^{ - 1}} \circ x} \right\|_{{\mathbb{H}^n}}^\alpha }}dy}  \in {C^\infty }\left( {{{\Omega '}_{{\raise0.5ex\hbox{$\scriptstyle \rho $}
\kern-0.1em/\kern-0.15em
\lower0.25ex\hbox{$\scriptstyle 8$}}}}} \right)
\end{equation}
and
\begin{equation}\label{eq527}
 {G_{Q + 2s}}\left[ {\left| \omega  \right|} \right]\left( x \right) = \int_{{\mathbb{H}^n}} {\frac{{\left| {\omega \left( y \right)} \right|}}{{\left\| {{y^{ - 1}} \circ x} \right\|_{{\mathbb{H}^n}}^\alpha }}dy}  \in {C^\infty }\left( {{{\Omega '}_{{\raise0.5ex\hbox{$\scriptstyle \rho $}
\kern-0.1em/\kern-0.15em
\lower0.25ex\hbox{$\scriptstyle 8$}}}}} \right).
\end{equation}
Then, for any $\varphi  \in C_0^\infty \left( {\Omega ',\mathbb{R}} \right)$, we apply \eqref{eq523} and \eqref{eq527} to get
\begin{align*}
   & \iint_{\mathbb{H}^n \times \mathbb{H}^n} {\frac{{\left| {\omega \left( x \right) - \omega \left( y \right)} \right|\left| {\varphi \left( x \right) - \varphi \left( y \right)} \right|}}{{\left\| {{y^{ - 1}} \circ x} \right\|_{{\mathbb{H}^n}}^{Q + 2s}}}dxdy}  \\
\le   &  2\iint_{\mathbb{H}^n \times \mathbb{H}^n} {\frac{{\left| {\omega \left( x \right) - \omega \left( y \right)} \right|}}{{\left\| {{y^{ - 1}} \circ x} \right\|_{{\mathbb{H}^n}}^{Q + 2s}}}\left| {\varphi \left( x \right)} \right|dxdy} \\
 = &2\int_{{\mathbb{H}^n}} {\left( {\int_{{\mathbb{H}^n}} {\frac{{\left| {\omega \left( x \right) - \omega \left( y \right)} \right|}}{{\left\| {{y^{ - 1}} \circ x} \right\|_{{\mathbb{H}^n}}^{Q + 2s}}}dy} } \right)\left| {\varphi \left( x \right)} \right|dx} \\
  = &2\int_{\Omega '} {\left( {\int_{{\mathbb{H}^n}} {\frac{{\left| {\omega \left( y \right)} \right|}}{{\left\| {{y^{ - 1}} \circ x} \right\|_{{\mathbb{H}^n}}^{Q + 2s}}}dy} } \right)\left| {\varphi \left( x \right)} \right|dx} \\
   = &2\int_{\Omega '} {{G_{Q + 2s}}\left[ {\left| \omega  \right|} \right]\left( x \right)\left| {\varphi \left( x \right)} \right|dx} \\
    \le& 2\left| {\Omega '} \right|\mathop {\max }\limits_{x \in \bar \Omega '} \left( {{G_{Q + 2s}}\left[ {\left| \omega  \right|} \right]\left( x \right)} \right)\mathop {\max }\limits_{x \in {\mathbb{H}^n}} \left| {\varphi \left( x \right)} \right| < \infty ,
\end{align*}
so it gives by using Fubini's theorem that
\begin{align*}
   & \iint_{\mathbb{H}^n \times \mathbb{H}^n} {\frac{{\left( {\omega \left( x \right) - \omega \left( y \right)} \right)\left( {\varphi \left( x \right) - \varphi \left( y \right)} \right)}}{{\left\| {{y^{ - 1}} \circ x} \right\|_{{\mathbb{H}^n}}^{Q + 2s}}}dxdy}  \\
  = &  2\iint_{\mathbb{H}^n \times \mathbb{H}^n} {\frac{{\left( {\omega \left( x \right) - \omega \left( y \right)} \right)}}{{\left\| {{y^{ - 1}} \circ x} \right\|_{{\mathbb{H}^n}}^{Q + 2s}}}\varphi \left( x \right)dxdy} \\
   = &2\int_{\Omega '} {{G_{Q + 2s}}\left[ \omega  \right]\left( x \right)\varphi \left( x \right)dx} .
\end{align*}
Substituting the above equation into \eqref{eq525} yields
\begin{align}\label{eq528}
   \int_{\Omega '} {f\varphi dx}  =& \int_{\Omega '} {{\nabla _H}\upsilon  \cdot {\nabla _H}\phi dx}  + \frac{{{c_{Q,s}}}}{2}\iint_{{\mathbb{H} ^n} \times {\mathbb{H} ^n}} {\frac{{\left( {\upsilon \left( x \right) - \upsilon \left( y \right)} \right)\left( {\varphi \left( x \right) - \varphi \left( y \right)} \right)}}{{\left\| {{y^{ - 1}} \circ x} \right\|_{{\mathbb{H} ^n}}^{Q + 2s}}}dxdy} \nonumber \\
   &  + {c_{Q,s}}\int_{\Omega '} {{G_{Q + 2s}}\left[ \omega  \right]\left( x \right)\varphi \left( x \right)dx} .
\end{align}
Hence it gains
\begin{align*}
   & \int_{\Omega '} {{\nabla _H}\upsilon  \cdot {\nabla _H}\phi dx}  + \frac{{{c_{Q,s}}}}{2}\iint_{\mathbb{H}^n \times \mathbb{H}^n} {\frac{{\left( {\upsilon \left( x \right) - \upsilon \left( y \right)} \right)\left( {\varphi \left( x \right) - \varphi \left( y \right)} \right)}}{{\left\| {{y^{ - 1}} \circ x} \right\|_{{\mathbb{H}^n}}^{Q + 2s}}}dxdy}  \\
 =  & \int_{\Omega '} {\left( {f - {c_{Q,s}}{G_{Q + 2s}}\left[ \omega  \right]} \right)\varphi dx}
\end{align*}
for any $\varphi  \in C_0^\infty \left( {\Omega ',\mathbb{R}} \right)$, which $\upsilon $ is a weak solution of \eqref{eq522} with $\psi : = {c_{Q,s}}{G_{Q + 2s}}\left[ \omega  \right]$, and $\psi  \in {C^\infty }\left( {\bar \Omega ',\mathbb{R}} \right)$ from \eqref{eq526}.

Now let us prove Theorem \ref{Th13} by using Proposition \ref{Pro52} and Lemma \ref{Le54}.

\textbf{Proof of Theorem \ref{Th13}.} Similar to the proof of Proposition \ref{Pro52}, we proceed by induction on $m \in  \mathbb{N}\cup \left\{ 0 \right\}$. For the case of $m=0$ (that is $f \in L ^ 2 (\Omega) $), see Theorem \ref{Th12}. Now we assume that Theorem \ref{Th13} holds for some $m>0$, and we prove that it still holds for $m+1$.

Let $V$ be a fixed open set with $V \subseteq \bar V \subseteq \Omega $, and $f \in H{W^{m + 1,2}}\left( \Omega  \right)$, $u \in H{W^{1,2}}\left( {{\mathbb{H}^n}} \right)$ be a weak solution of \eqref{eq11}. Specifically, $f \in H{W^{m,2}}\left( \Omega  \right)$, then $u \in HW_{\mathrm{loc}}^{m + 2,2}\left( \Omega  \right)$ from the inductive assumption, and thus
\begin{equation}\label{eq529}
u \in H{W^{m + 2,2}}\left( V \right).
\end{equation}

For any fixed open set $\Omega '$ with $\bar V \subseteq \Omega ' \subseteq \bar \Omega ' \subseteq \Omega $, we set
\[\rho  = {\rm{dist}}\left( {\Omega ',\partial \Omega } \right) > 0.\]
Furthermore, we take a truncation function $\xi  \in C_0^\infty \left( {{\mathbb{H}^n},\mathbb{R}} \right)$ that satisfies (1)-(3) in Lemma \ref{Le54} and define $\upsilon : = u\xi $. Then, we have
\[\upsilon  \in H{W^{m + 2,2}}\left( {{\mathbb{H}^n}} \right)\]
from \eqref{eq529} and ${\rm{supp}}\xi  \subseteq \Omega $, so by Lemma \ref{Le54}, there exists $\psi  \in {C^\infty }\left( {\bar \Omega ',\mathbb{R}} \right)$ such that $\upsilon $ is a weak solution to the equation
\begin{equation}\label{eq530}
   - \Delta_{\mathbb{H}^n} \upsilon \left( x \right) + {c_{Q,s}}P.V.\int_{{\mathbb{H}^n}} {\frac{{\upsilon \left( x \right) - \upsilon \left( y \right)}}{{\left\| {{y^{ - 1}} \circ x} \right\|_{{\mathbb{H}^n}}^{Q + 2s}}}dy}  = f\left( x \right) + \psi \left( x \right),\;x \in \Omega '.
\end{equation}

Noting that $f \in H{W^{m + 1,2}}\left( \Omega  \right)$, $\psi  \in {C^\infty }\left( {\bar \Omega ',} \mathbb{R}\right)$, it follows $\tilde f: = f + \psi  \in H{W^{m + 1,2}}\left( \Omega  \right)$. Then we apply Proposition \ref{Pro52} to deduce $\upsilon  \in HW_{\mathrm{loc}}^{m + 3,2}\left( \Omega  \right)$. In particular, $\upsilon  \in H{W^{m + 3,2}}\left( V \right).$ Again because $\xi  \equiv 1$ on ${\Omega '_{{\raise0.5ex\hbox{$\scriptstyle \rho $}
\kern-0.1em/\kern-0.15em
\lower0.25ex\hbox{$\scriptstyle 4$}}}} \supset V$, we have
\[\upsilon  \equiv u \in H{W^{m + 3,2}}\left( V \right).\]
Therefore, Theorem \ref{Th13} is proved.

\section*{Declarations}
\subsection*{Funding}
The first author was supported by the National Natural Science Foundation of China (No. 12501269) and the Scientific Research Program Funded by Shaanxi Provincial Education Department  (No. 25JK03\\60).





\subsection*{Data Availability}
Data sharing is not applicable to this article as no datasets were generated or analysed during the current study.

\subsection*{Conflict of interest} The authors declare that there is no conflict of interest. We also declare that this manuscript has no associated data.



\begin{thebibliography}{99}
\bibitem{AM18} A. Adimurthi, A. Mallick, Hardy type inequality on fractional order Sobolev spaces on the Heisenberg
group, \emph{Ann. Sc. Norm. Super. Pisa Cl. Sci.}, \textbf{18(3)}, (2018), 917-949.

\bibitem{BFM13} T. P. Branson, L. Fontana, C. Morpurgo, Moser-Trudinger and Beckner-Onofri's Inequalities on the CR sphere, \emph{Ann. of Math.}, \textbf{177(1)}, (2013), 1-52.

\bibitem{BDVV22} S. Biagi, S. Dipierro, E. Valdinoci, et al., Mixed local and nonlocal elliptic operators: regularity and maximum principles,
\emph{Comm. Partial Differential Equations}, \textbf{47(3)}, (2022), 585-629.

\bibitem{BMS23} A. Biswas, M. Modasiya, A. Sen, Boundary regularity of mixed local-nonlocal operators and its application,
\emph{Ann. Mat. Pur. Appl.}, \textbf{202(2)}, (2023), 679-710.

\bibitem{CL97} L. Capogna, Regularity of quasi-linear equations in the Heisenberg group, \emph{Comm. Pure Appl. Math.}, \textbf{50(9)}, (1997), 867-889.

\bibitem{CL99} L. Capogna, Regularity for quasilinear equation and 1-quasiconformal maps in carnot groups, \emph{Math. Ann.}, \textbf{313(2)}, (1999), 263-295.

\bibitem{CCR15} P. Ciatti, M. G. Cowling, F. Ricci, Hardy and uncertainty inequalities on stratified Lie groups, \emph{Adv. Math.}, \textbf{277}, (2015), 365-387.

\bibitem{CT16} E. Cinti, J. Tan, A nonlinear Liouville theorem for fractional equations in the Heisenberg group, \emph{J. Math. Anal. Appl.}, \textbf{433}, (2016), 434-454.

\bibitem{DM20} C. De Filippis, G. Mingione, On the regularity of non-autonomous functionals, \emph{J. Geom. Anal.}, \textbf{30(2)}, (2020), 1584-1626.

\bibitem{DM24} C. De Filippis, G. Mingione, Gradient regularity in mixed local and nonlocal problems,
\emph{Math. Ann.}, \textbf{388(1)}, (2024), 261-328.

\bibitem{DA04} A. Domokos, Differentiability of solutions for the non-degenerate $p$-Laplacian in the Heisenberg group, \emph{J. Differential Equations}, \textbf{204(2)}, (2004), 439-470.

\bibitem{DM05} A. Domokos, J. Manfredi, ${C^{1,\alpha }}$-regularity for $p$-harmonic functions in the Heisenberg group for $p$ near 2, The $p$-harmonic equation and recent advances in analysis, 17-23, \emph{Contemp. Math.}, \textbf{370}, Amer. Math. Soc., Providence, RI, 2005.

\bibitem{DFZ24} M. Ding, Y. Fang, C. Zhang, Local behavior of the mixed local and nonlocal problems with nonstandard growth,
\emph{J. Lond. Math. Soc.}, \textbf{109(6)}, (2024), 1-28.

\bibitem{FZZ24} Y. Fang, C. Zhang, J. Zhang, Local regularity for nonlocal double phase equations in the Heisenberg group, \emph{Proc. Roy. Soc. Edinburgh Sect. A}, 2024, https://doi.org/10.1017/prm.2024.89.

\bibitem{FF15} F. Ferrari, B. Franchi, Harnack inequality for fractional Laplacians in Carnot groups, \emph{Math. Z.}, \textbf{279}, (2015), 435-458.

\bibitem{FMPPS18} F. Ferrari, M. Jr. Miranda, D. Pallara, A. Pinamonti, Y. Sire, Fractional Laplacians, perimeters and heat semigroups in Carnot groups, \emph{Discrete Cont. Dyn. Syst. Ser. S}, \textbf{11(3)}, (2018), 477-491.

\bibitem{FGMT15} R. Frank, M. Gonzalez, D. Monticelli, J. Tan, An extension problem for the CR fractional Laplacian, \emph{Adv. Math.}, \textbf{270}, (2015), 97-137.

\bibitem{GL23} P. Garain, E. Lindgren, Higher H\"{o}lder regularity for mixed local and nonlocal degenerate elliptic equations,
\emph{Calc. Var. Partial Differential Equations}, \textbf{62(2)}, (2023), 1-36.

\bibitem{GT22} N. Garofalo, G. Tralli, A class of nonlocal hypoelliptic operators and their extensions, \emph{Indiana Univ. Math. J.}, \textbf{70(5)}, 2022, 1717-1744.

\bibitem{GT21} N. Garofalo, G. Tralli, Feeling the heat in a group of Heisenberg type, \emph{Adv. Math.}, \textbf{381}, (2021), Art. 107635.

\bibitem{HL67} L. H\"{o}rmander, Hypoelliptic second order differential equations, \emph{Acta Math.}, \textbf{119(3)}, (1967), 147-171.

\bibitem{KS18} A. Kassymov, D. Surgan, Some functional inequalities for the fractional $p$-sub-Laplacian, arXiv:1804.01415.

\bibitem{KS20} A. Kassymov, D. Suragan, Lyapunov-type inequalities for the fractional $p$-sub-Laplacian, \emph{Adv. Oper. Theory}, \textbf{5(2)}, (2020), 435-452.

\bibitem{MM07} J.J. Manfredi, G. Mingione, Regularity results for quasilinear elliptic equations in the Heisenberg Group, \emph{Math. Ann.}, \textbf{339(3)}, (2007), 485-544.

\bibitem{MPP23} M. Manfredini, G. Palatucci, M. Piccinini, S. Polidoro, H\"{o}lder continuity and boundedness estimates for nonlinear fractional equations in the Heisenberg group, \emph{J. Geom. Anal.}, \textbf{33(3)}, (2023), Art 77.

\bibitem{MS03} S. Marchi, ${C^{1,\alpha }}$ local regularity for the solutions of the $p$-Laplacian on the Heisenberg group for $2 \le p \le 1 + \sqrt 5 $, \emph{Z. Anal. Anwendungen}, \textbf{20(3)}, (2001), 617-636. \emph{Z. Anal. Anwendungen}, \textbf{22(2)}, (2003), 471-472.

\bibitem{MS031} S. Marchi, ${C^{1,\alpha }}$ local regularity for the solutions of the $p$-Laplacian on the Heisenberg group, The case $1 + \frac{1}{{\sqrt 5 }} < p \le 2$, \emph{Comment. Math. Univ. Carolin.}, \textbf{44(1)}, (2003), 33-56.


\bibitem{MZZ07} G. Mingione, A. Zatorska-Goldstein, X. Zhong, Gradient regularity for elliptic equations in the Heisenberg group, \emph{Adv. Math.}, \textbf{222(1)}, (2009), 62-129.

\bibitem{MZ17} S. Mukherjee, X. Zhong, ${C^{1,\alpha }}$-Regularity for variational problems in the Heisenberg group, \emph{Anal. PDE}, \textbf{14(2)}, (2021), 567-594.

\bibitem{OT25} P. Oza, J. Tyagi, Existence and regularity of solutions to mixed fully nonlinear local and nonlocal sub-elliptic equation in the Heisenberg group, \emph{Anal. Math. Phys.}, \textbf{15(1)}, (2025), Art. 18.

\bibitem{PP22} G. Palatucci, M. Piccinini, Nonlocal Harnack inequalities in the Heisenberg group, \emph{Calc. Var. Partial Differential Equations}, \textbf{61(5)}, (2022), Art. 185.

\bibitem{PP23} G. Palatucci, M. Piccinini, Nonlinear fractional equations in the Heisenberg group, \emph{Bruno Pini Math. Anal.}, \textbf{14(2)}, (2023), 163-200.

\bibitem{P22} M. Piccinini, The obstacle problem and the Perron Method for nonlinear fractional equations in the Heisenberg group, \emph{Nonlinear Anal.}, \textbf{222}, (2022), Art. 112966.

\bibitem{R15} D. Ricciotti, $p$-Laplace equation in the Heisenberg group, SpringerBriefs in Mathematics. Springer, BCAM
Basque Center for Applied Mathematics, Bilbao, 2015

\bibitem{R18} D. Ricciotti, On the ${C^{1,\alpha }}$-regularity of $p$-harmonic functions in the Heisenberg group, \emph{Proc. Amer. Math. Soc.}, \textbf{146(7)}, (2018), 2937-952.

\bibitem{RT16} L. Roncal, S. Thangavelu, Hardy's Inequality for fractional powers of the sublaplacian on the Heisenberg group, \emph{Adv. Math.}, \textbf{302}, (2016), 106-158.

\bibitem{SVWZ22} X. Su, E. Valdinoci, Y. Wei, et al., Regularity results for solutions of mixed local and nonlocal elliptic equations, \emph{Math. Z.}, \textbf{302(3)}, (2022), 1855-1878.

\bibitem{ZHT23} X. Zha, S. Huang, Q. Tian, Uniform boundedness results of solutions to mixed local and nonlocal elliptic operator,
\emph{Aims Math.}, \textbf{8(9)}, (2023), 20665-20678.

\bibitem{Z26} J. Zhang, $C^{1,\alpha}$-regularity for Mixed Local and Nonlocal Degenerate Elliptic Equations in the Heisenberg Group. arXiv preprint arXiv:2602.08398. https://doi.org/10.48550/arXiv.2602.08398

\bibitem{ZN20} J. Zhang, P. Niu, H\"{o}lder regularity of quasiminimizers to generalized Orlicz functional on the Heisenberg group, \emph{J. Funct. Spaces}, \textbf{Art. ID 8838654}, (2020), Art. 13.

\bibitem{ZN22} J. Zhang, P. Niu, Fractional estimates for nonlinear non-differentiable degenerate elliptic system on the Heisenberg group, \emph{J. Nonlinear Convex Anal.}, \textbf{23}, (2022), 1419-1451.

\bibitem{ZL23} J. Zhang, Z. Li, Weak Differentiability to Nonuniform Nonlinear Degenerate Elliptic Systems under $P,q$-growth Condition on the Heisenberg Group, \emph{J. Math. Anal. Appl.}, \textbf{525}, (2023), Art. 42.

\bibitem{ZN26} J. Zhang, P. Niu, Regularity for Mixed Local and Nonlocal Degenerate Elliptic Equations in the Heisenberg Group, \emph{J. Differential Equations}, \textbf{453}, (2026), Art. 113888.


\end{thebibliography}
\end{document}